\documentclass{ws-m3as}

\usepackage{graphicx,epsfig}
\usepackage{epstopdf}
\usepackage{pdfpages}
\usepackage{amssymb}
\usepackage{empheq}
\usepackage{cases}
\usepackage{amsmath}
\usepackage{caption,lipsum}
\usepackage{stmaryrd}
\usepackage{tabularx}
\usepackage{color}
\usepackage{extarrows}
\usepackage{empheq,float}
\DeclareGraphicsExtensions{.eps}

\usepackage{epstopdf}
\usepackage{pdfpages}
\usepackage{amsfonts,amsmath,amssymb}
\usepackage{graphicx,epsfig}
\usepackage{graphicx,epsfig,float,enumitem}
\usepackage{xcolor}
\usepackage{soul}
\usepackage{pgfplots}
\usepackage[hidelinks]{hyperref}
\usepackage[capitalize,nameinlink,noabbrev]{cleveref}
\crefname{equation}{}{}
\usepackage[normalem]{ulem}

\DeclareMathOperator*{\argmin}{arg\,min}
\newcommand{\lrang}[2]{\langle #1, #2 \rangle}
\newcommand{\R}{\mathbb{R}}
\newcommand{\vx}{\mathbf{x}}
\newcommand{\rmd}{\mathrm d}
\newtheorem{assumption}[theorem]{Assumption}

\numberwithin{equation}{section}
\numberwithin{figure}{section}
\numberwithin{table}{section}

\begin{document}

\markboth{W. Liu, C. Wang, X. Zhao}{Action ground states of defocusing nonlinear Schr\"odinger equations}

%
%

\title{\large On action ground states of defocusing nonlinear Schr\"odinger equations}

\author{\normalsize Wei Liu}

\address{Department of Mathematics, National University of Defense Technology, Changsha 410073, China \& Department of Mathematics, National University of Singapore, Singapore 119076, Singapore\\
wl@nudt.edu.cn}

\author{\normalsize Chushan Wang\footnote{Corresponding author}}

\address{Department of Mathematics, National University of Singapore, Singapore 119076, Singapore\\
E0546091@u.nus.edu}

\author{\normalsize Xiaofei Zhao}

\address{School of Mathematics and Statistics \& Computational Sciences Hubei Key Laboratory, Wuhan University\\ 
Wuhan 430072, China\\
matzhxf@whu.edu.cn}

\maketitle


\begin{abstract}
We investigate the action ground states of the defocusing nonlinear Schr\"odinger equation with and without rotation. Our primary focus is on characterizing the relationship between the action ground states and the energy ground states. Theoretically, we prove a complete equivalence of the two in the non-rotating case and a conditional equivalence in the rotating case. Our theoretical results are supported by extensive numerical experiments. Notably, in the rotating case, we provide numerical examples of non-equivalence showing that non-equivalence typically occurs at the transition points where the number of vortices in the action ground state is increasing. Additionally, we study the asymptotic behaviour of the action ground states and the associated physical quantities in certain limiting parameter regimes, with numerical results validating and complementing our analysis. Furthermore, we explore the formation and change of the vortex pattern in the action ground states numerically.
\end{abstract}

\keywords{defocusing nonlinear Schr\"odinger equation; action ground state; rotation; energy ground state;
asymptotics; quantized vortices.}

\ccode{AMS Subject Classification: 35B38, 35Q55, 58E30, 81-08}

\section{Introduction}
The rotating nonlinear Schr\"odinger equation (RNLS) is widely adopted in various physical applications, such as quantum physics, nonlinear optics, and Bose-Einstein condensation (BEC). As a general model, we consider the following RNLS as
\begin{equation}\label{model}
i\partial_t\psi = -\frac{1}{2}\Delta \psi + V(\mathbf{x}) \psi + \beta|\psi|^{p-1}\psi - \Omega L_z \psi, \quad t>0, \quad \mathbf{x} \in \mathbb{R}^d,
\end{equation}
where $1<p<\frac{d+2}{d-2}$ ($1<p<\infty$ when $d=1, 2$), $\Omega,\beta\in\mathbb{R}$ are given parameters and $\psi : = \psi(\mathbf{x}, t):\mathbb{R}^d \times \R \to \mathbb{C}$ is the wave function with
$\mathbf{x}=(x_1,\ldots,x_d)^\top$ and $d\in\mathbb{N}_+$. Here, $V(\mathbf{x})$ is a given real-valued potential function and $L_z$ is the so-called angular momentum operator defined as
\begin{equation*}
L_z=\left\{
\begin{aligned}
&i(x_2\partial_{x_1}-x_1\partial_{x_2}), && d\geq2,\\
&0, && d=1.
\end{aligned}
\right.
\end{equation*}
The parameter $\Omega$ is interpreted as the rotating speed. When $\Omega=0$, \cref{model} reduces to the classical NLS with an external potential which is a fundamental model in quantum physics and nonlinear optics. When $\Omega \neq 0$, $p=3$ and $d=2,3$, \cref{model} serves as the mean-field model for the dynamics of BEC in a rotating frame\cite{BaoCai,Baoadd0,Baoadd1,Fetter}. When $d \geq 3$, the restriction $p<\frac{d+2}{d-2}$ (i.e., energy-subcritical nonlinearity) is imposed to ensure the well-posedness of the equation\cite{Antonelli,Lions1}. The parameter $\beta$ characterizes the type and strength of the nonlinear interaction, with $\beta > 0$ corresponding to a defocusing interaction and $\beta < 0$ to a focusing interaction. Both defocusing and focusing cases are widely considered in applications. Under such setup, the mass $m=\|\psi(\cdot,t)\|_{L^2}^2$ and the energy
\begin{equation}\label{energy fun}
E(\psi(\cdot,t)):= \frac{1}{2}\|\nabla \psi\|_{L^2}^2 + \int_{\mathbb{R}^d} V |\psi|^2 \rmd\vx + \frac{2\beta}{p+1}\|\psi\|_{L^{p+1}}^{p+1} + L_\Omega(\psi)
\end{equation}
are preserved by \eqref{model} in time $t$, where
\[ L_\Omega(\psi):=-\Omega\int_{\mathbb{R}^d} \text{Re}(\overline{\psi}(\mathbf{x}, t)L_z\psi(\mathbf{x}, t))\rmd\mathbf{x} = -\Omega\int_{\mathbb{R}^d} \overline{\psi}(\mathbf{x}, t)L_z\psi(\mathbf{x}, t) \rmd\mathbf{x} \]
by integration by parts. We keep Re in the definition since it may be necessary for numerical approximations.

As a special class of solutions for dynamics, standing wave/stationary solutions that maintain their shape are of special interest in various important applications.  Examples include the solitary impulse in optical communications and the ground or excited states in cold atoms.
With a prescribed frequency $\omega\in\mathbb{R}$ ($-\omega$ referred to as chemical potential in quantum physics), a standing wave solution of \cref{model} takes the form
$\psi(\mathbf{x},t)=\mathrm{e}^{i\omega t}\phi(\mathbf{x})$, where $\phi$ solves the stationary RNLS
\begin{equation}\label{model0}
H(\phi):=-\frac{1}{2}\Delta\phi+V\phi+\beta |\phi|^{p-1}\phi-
\Omega L_z\phi+\omega\phi=0.
\end{equation}
The fact that \cref{model0} can admit infinitely many solutions\cite{Lions2,Strauss} motivates different mathematical ways to define the physically relevant ones among all the nontrivial solutions. In this work, we are concerned with the so-called \emph{action ground state} which is defined as non-trivial solutions that minimize the \emph{action functional}
\begin{align}\label{action}
S_{\Omega,\omega}(\phi)
&:= E(\phi) + \omega \|\phi\|_{L^2}^2 \notag \\
&= \frac{1}{2}\|\nabla \phi\|_{L^2}^2
 +\int_{\mathbb{R}^d}V|\phi|^2\rmd\mathbf{x}
 +\frac{2\beta}{p+1}
 \|\phi\|_{L^{p+1}}^{p+1}+L_\Omega(\phi)+\omega\|\phi\|_{L^2}^2.
\end{align}
That is to say, we are interested in the solution
\begin{equation}\label{phi_g-def}
	\phi_g = \phi_{\Omega, \omega} \in \mathcal{A}_{\Omega, \omega} := \argmin\{S_{\Omega,\omega}(\phi) :  H(\phi)=0, \  \phi \neq 0\}.
\end{equation}
As we shall show later, the restriction $\phi \neq 0$ in \cref{phi_g-def}, which is needed in the focusing case, can be dropped in the defocusing case. The quantity (\ref{action}) is indeed representing the physical action (the space-time integration of Lagrangian function) of the system for the stationary solution. Such definition of ground states, as far as we know, can date back to the early work of Berestycki and Lions\cite{Lions1}, and it has been further developed in subsequent mathematical research, e.g., Refs.~\refcite{Hajaiej,Fukuizumi,Fukuizumi2,Ohta,Shatah,Wang}. While most studies in the literature concerns the action ground states in the case of focusing nonlinearity, the defocusing case has received comparatively little attention. In our recent work Ref.~\refcite{LYZ}, we address the defocusing case of \cref{model0}, where an equivalent unconstrained variational formulation of the action ground state is identified with efficient numerical algorithms proposed to compute the action ground state solutions. However, many mathematical aspects of the defocusing action ground states remain unexplored, which is one of the goals of our work.

In fact, a more common definition of the ground state in the literature is introduced as the minimization of the energy functional $E$ under a prescribed mass $m>0$\cite{Sparber,BaoCai,Baoadd1,Seiringer2002CMP}:
\begin{align}\label{energyGS-def}
\phi_m^E\in \mathcal{B}_{\Omega, m} := \argmin\{ E(\phi) : \|\phi\|_{L^2}^2=m, \  E(\phi)<\infty \},
\end{align}
which is referred to as the {\em energy ground state}. The prescription of mass is preferred among physicists, and so the energy ground states have received extensive investigations in the past two decades for both focusing ($d=1,2$) and defocusing cases ($d=1,2,3$). See Refs.~\refcite{BaoCai,CDLX,LiuCai21,Seiringer2002CMP} and the references therein. Note that $\phi_m^E$ is also a stationary solution satisfying \eqref{model0}, and the corresponding chemical potential, playing the role of Lagrange multiplier, is determined  as
\[ \omega(\phi_m^E)=-\mu(\phi_m^E), \]
where
\begin{align}\label{eq:mu_def}
    \mu(\phi)
    :&=\frac{1}{\|\phi\|_{L^2}^2}\bigg[\frac{1}{2}\|\nabla \phi\|_{L^2}^2 + \int_{\mathbb{R}^d}V|\phi|^2\rmd\vx + \beta \|\phi\|_{L^{p+1}}^{p+1} + L_\Omega(\phi)\bigg] \notag \\
    &=\frac{1}{\|\phi\|_{L^2}^2}\bigg[E(\phi)+\frac{p-1}{p+1}\beta\|\phi\|_{L^{p+1}}^{p+1}\bigg].
\end{align}
A long-standing question pertains to the relationship between the action ground state and the energy ground state. More interestingly, it raises the question that if the following diagram can commute:
\begin{align}\label{diagram}
\begin{array}{rcl}
 \omega ~~~& \xlongrightarrow{~~~\cref{phi_g-def}~~~} & ~\phi_{\Omega, \omega} \\[1.5ex]
 \omega = -\mu(\phi_m^E) \quad \rotatebox[origin=c]{90}{$\displaystyle\dashrightarrow$}~~~ &  & ~~~\rotatebox[origin=c]{-90}{$\displaystyle\dashrightarrow$}\quad m=\|\phi_{\Omega, \omega}\|_{L^2}^2 \\[1.5ex]
 \phi_m^E ~& \xlongleftarrow[~~~\cref{energyGS-def}~~~]{} &~~~ m \\
\end{array},
\end{align}
which is important both mathematically and physically. Recently, some numerical investigations have been presented in Refs.~\refcite{LYZ,Wang}, and some theoretical answers have been provided in Refs.~\refcite{Dovetta,Jeanjean2021}  for the focusing and non-rotating case ($\beta<0,\Omega=0$) of \cref{model}.


In this work, we investigate the action ground states of the defocusing RNLS (i.e., $\beta>0$ afterward). We first focus on characterizing the relationship between the action ground states and the energy ground states. Then we study the asymptotical behaviours of the action ground state solutions and the associated physical quantities including mass and action in some limiting regimes of $\omega$ and $\Omega$. For the convenience of the reader, we summarize our main results in these two directions below.

In terms of the relationship between the action and energy ground states:
\begin{itemize}
	\item In the non-rotating case (i.e., $\Omega = 0$), we prove a complete equivalence between action and energy ground states: any action ground state is also an energy ground state (\cref{thm:actiontoenergy}) and any energy ground state is also an action ground state (\cref{thm:energytoaction}).
	
	\item In the rotating case (i.e., $\Omega \neq 0$), we establish both equivalence and non-equivalence results: any action ground state $\phi_{\Omega, \omega}$ is an energy ground state with mass $m = \| \phi_{\Omega, \omega} \|^2_{L^2}$, and all the energy ground states with this mass $m$ share the same chemical potential and are also action ground states (\cref{thm:actiontoenergy_new}); however, an energy ground state with mass $m$ will not be an action ground state in the situation stated in \cref{thm:Non-equivalence}. An alternative description of the equivalence via the dual problem is given in \cref{thm:dual}.
\end{itemize}
\noindent The proof is mainly based on the unconstrained formulation from Ref. \refcite{LYZ} and the framework of arguments from Ref. \refcite{Dovetta}. Here, the main difference between the rotating case and the non-rotating case lies in whether there is uniqueness of both ground states. For the non-rotating case, due to the uniqueness of non-negative action ground states and energy ground states, we can establish a complete equivalence between the two and the diagram \cref{diagram} in this case forms a closed loop. In comparison, in the rotating case, we no longer have such uniqueness, which poses much greater challenges in determining if an energy ground state is also an action ground state. In this regard, we derive a sufficient and necessary condition to characterize when the non-equivalence would occur. Although such characterization cannot prove or disprove the equivalence, it has guided our numerical experiments, where examples of non-equivalence are discovered. Moreover, the numerical results suggest that the non-equivalence happens exactly at the transition point $\omega$ where the number of vortices in action ground states changes.

In terms of the asymptotical behaviours of action ground states:
	\begin{itemize}
		\item For given $\omega$, as the rotation speed $\Omega \rightarrow 0$, the action ground state with rotation will converge to the ground state in the non-rotating case (\cref{thm:defocusing}).
		
		\item As $\omega \rightarrow \lambda_0(\Omega)$ with $\lambda_0(\Omega)$ being the smallest eigenvalue of the linear Schr\"odinger operator, the normalized action ground state will converge to the ground state of the linear operator (\cref{thm:smallomega_defocusing}) with sharp estimates of the convergence rate (\cref{thm:smallomega_asymptotics}).
		
		\item As $\omega \rightarrow -\infty$, we formally derive a Thomas-Fermi type limit in \cref{subsec:omega infinity}.
\end{itemize}
\noindent All these asymptotical results are validated and complemented by numerical studies, which further explored {the formation and change of the vortex pattern in the action ground states in the rotating case}. In particular, we find the critical rotating speed $\Omega^\text{c}$ for the generation of vortices and examine how $\Omega^\text{c}$ is influenced by $\omega$ as well as investigate how the number of vortices changes with respect to $\Omega$ and $\omega$.

The rest of the paper is structured as follows. Some  preliminaries for analysis will be given in \cref{sec:pre}. The complete equivalence in the non-rotating case will be  established in \cref{sec:3}. Then, the general rotating case will be studied in \cref{sec:rot}, where the conditional equivalence and non-equivalent situations as well as some asymptotic results are theoretically characterized. Finally,  numerical explorations are done in \cref{sec:numer} and some concluding remarks are made in \cref{sec:con}.

\section{Preliminary}\label{sec:pre}
In this section, we present some basic notations, setups and facts and review some useful results that will be called for investigation later.

\subsection{Setup and assumption}
We assume $V(\mathbf{x})\geq0$ and introduce the functional spaces
\[ L_V^2(\mathbb{R}^d) := \left\{ \phi\ :\ \int_{\mathbb{R}^d} V(\mathbf{x})|\phi(\mathbf{x})|^2 \rmd\mathbf{x}<\infty \right\}, \quad
X:=H^1(\mathbb{R}^d)\cap L_V^2(\mathbb{R}^d).  \]
Clearly, $X$ equipped with the following inner product is a Hilbert space:
\[ (u,v)_X:= \int_{\mathbb{R}^d} \Big(\nabla u(\mathbf{x})\cdot\nabla\overline{v(\mathbf{x})} + \big(1+V(\mathbf{x})\big)u(\mathbf{x})\overline{v(\mathbf{x})} \Big) \rmd\mathbf{x}, \quad\forall\ u,v\in X. \]
We further assume that $V$ satisfies the confining condition
\begin{equation*}
	\lim_{|\vx| \rightarrow \infty}V(\vx) = \infty,
\end{equation*}
which yields the following compact embeddings.
\begin{lemma}[Compact embedding\cite{BaoCai}]\label{lem:Xembed}
Assume that $V(\mathbf{x})\geq0 \; (\forall\ \mathbf{x}\in\mathbb{R}^d)$ satisfies the confining condition. Then the embedding $X\hookrightarrow L^q(\mathbb{R}^d)$ is compact, where $q\in[2,\infty]$ for $d=1$, $q\in[2,\infty)$ for $d=2$, and $q\in[2,2d/(d-2))$ for $d\geq3$.
\end{lemma}

When $d \geq 2$, by integration by parts, the angular momentum $L_\Omega(\phi) = -\Omega \int_{\R^d} \overline{\phi(\vx)} L_z \phi(\vx) \rmd \vx$ is real-valued for any $\phi \in X$, and $L_\Omega(\phi) = 0$ if $\phi$ is real-valued. Moreover, the Young's inequality directly implies, for any $\delta>0$,
\begin{equation}\label{eq:rotenergy}
\left| \Omega\int_{\mathbb{R}^d} \overline{\phi}L_z\phi\, \rmd\mathbf{x}\right|
\leq \int_{\mathbb{R}^d} \left(\frac{\delta}{2}|\nabla\phi|^2 +\frac{|\Omega|^2}{2\delta}(x_1^2+x_2^2)|\phi|^2 \right) \rmd\mathbf{x},
\end{equation}
which motivates the following definition of the maximum rotating speed as
\begin{equation}\label{Omega max}
	\Omega_{\max}:=\sup\left\{\gamma \geq 0:\lim_{|\mathbf{x}|\to\infty}\left(V(\mathbf{x})-\frac{\gamma^2}{2}(x_1^2+x_2^2)\right)=\infty\right\}.
\end{equation}

Now we summarize all the assumptions for our problem below which will be assumed throughout the rest of this paper.
\begin{assumption}\label{lem:S-welldef}
Let $1<p<\frac{d+2}{d-2}$ for $d\geq3$ and $1<p<\infty$ for $d=1,2$. Assume that $V(\mathbf{x})\geq0$ satisfies the confining condition. For the rotating speed $\Omega$, we assume either $\Omega=0$ (non-rotating case) or $0<|\Omega|<\Omega_{\max}$ with an additional assumption that $\Omega_\text{max}>0$ (rotating case). 
\end{assumption}
Under the assumptions above, $S_{\Omega,\omega}$ \cref{action} is well-defined on $X$ and we have
\begin{equation}\label{lem:control_rotation}
    \left| \int_{\R^d} \overline{\phi} L_z \phi\, \rmd \vx \right| \lesssim \| \phi \|_X^2,
\end{equation}
where the constant depends exclusively on $V$. In fact, by Assumption \ref{lem:S-welldef}, for any $0< \gamma < \Omega_\text{max} $, there exists $r_0>0$ such that $V(\vx) \geq \frac{\gamma^2}{2}(x_1^2 + x_2^2)$, $|\vx| \geq r_0$. Then \cref{lem:control_rotation} follows from
\begin{align*}
    &\left| \int_{\mathbb{R}^d} \overline{\phi}L_z\phi\, \rmd\mathbf{x}\right|
\leq \int_{\mathbb{R}^d} \left[\frac{1}{2}|\nabla\phi|^2 + \frac{1}{2}(x_1^2+x_2^2)|\phi|^2 \right] \rmd \vx \notag \\
=&\int_{\mathbb{R}^d}\frac{1}{2}|\nabla\phi|^2\rmd\vx + \frac{1}{2}\int_{|\vx| \geq r_0}  (x_1^2+x_2^2)|\phi|^2 \rmd \vx + \frac{1}{2}\int_{|\vx| < r_0}  (x_1^2+x_2^2)|\phi|^2 \rmd \vx \nonumber \\
\leq& \int_{\mathbb{R}^d} \left[\frac{1}{2}|\nabla\phi|^2 + \frac{1}{\gamma^2}V(\vx)|\phi|^2 \right] \rmd \vx+ r_0^2\|\phi\|_{L^2}^2 \lesssim \| \phi \|_X^2.
\end{align*}

To simplify the presentation, we assume, without loss of generality, $\Omega \geq 0$ and $\beta = 1$ in the rest of the paper. The two parameters $\omega$ and $\Omega$ will decisively affect the features of the action ground state, and they will be mainly concerned in our study.

\subsection{Related functionals}
To study the variational properties of the problem, it is convenient to introduce some more functionals that are closely related to the action. 
\begin{enumerate}[label= \roman*)]

\item \emph{Quadratic functional:} By excluding the $L^{p+1}$-term in the action $S_{\Omega,\omega}$ (\ref{action}), a quadratic functional is defined as
	\begin{equation*}\begin{split}
		Q_{\Omega,\omega}(\phi) :=& \int_{\mathbb{R}^d}\left(\frac12 |\nabla \phi|^2 + V|\phi|^2 + \omega|\phi|^2 - \Omega \overline{\phi} L_z \phi \right)\mathrm{d} \mathbf{x}\\
=&S_{\Omega,\omega}(\phi)-\frac{2}{p+1}\|\phi\|_{L^{p+1}}^{p+1}, \quad \phi \in X.\end{split}
	\end{equation*}
\item \emph{Nehari functional:}
The Nehari functional reads
	\begin{equation}\label{eq:K_def}
		K_{\Omega,\omega} (\phi) := Q_{\Omega,\omega}(\phi) +\| \phi \|_{L^{p+1}}^{p+1}=S_{\Omega,\omega}(\phi) + \frac{p-1}{p+1} \| \phi \|_{L^{p+1}}^{p+1}, \quad \phi \in X
	\end{equation}
and it induces the \emph{Nehari manifold}
$\mathcal{M}:=\{\phi\in X\backslash\{0\}:K_{\Omega,\omega}(\phi)=0\}.$
\end{enumerate}
It is known\cite{Hajaiej,Fukuizumi,LYZ} that, the action ground state (\ref{phi_g-def}) can be equivalently written as the minimizer of \cref{action} on the Nehari manifold $\mathcal{M}$, i.e.,
$ \phi_g \in \mathcal{A}_{\Omega, \omega} \Leftrightarrow\phi_g \in \argmin\{S_{\Omega,\omega}(\phi) : \phi\in \mathcal{M}\} $. Moreover, as one key property for the defocusing action ground state, another equivalent variational characterization of the action ground state as an unconstrained action minimizer on $X$ was established in Ref.~\refcite{LYZ}, i.e.,
$ \phi_g \in \mathcal{A}_{\Omega, \omega} \Leftrightarrow \phi_g \in \argmin\{S_{\Omega,\omega}(\phi): \phi\in X\} $. In fact, our analysis of the action ground states will be mainly based on this new variational characterization, and we shall recall it in the following.

Let $R:= - \frac12\Delta + V - \Omega L_z$ be the linear Schr\"odinger operator with rotation and define $\lambda_0(\Omega)$ as the smallest eigenvalue of $R$ on $X$, i.e.,
\begin{equation}\label{inf:linear with rotation}
\lambda_0(\Omega) := \inf_{0 \neq \phi \in X} \frac{\int_{\mathbb{R}^d} \left(\frac12|\nabla \phi|^2 + V|\phi|^2 - \Omega \overline{\phi} L_z \phi \right) \mathrm{d}\mathbf{x}}{\| \phi \|_{L^2}^2}.
\end{equation}
\begin{proposition}[Unconstrained minimization\cite{LYZ}]\label{thm:exist}
Let $\omega<-\lambda_0(\Omega)$. Under Assumption \ref{lem:S-welldef}, there exists a $\phi_g\in X$ such that
\begin{align}\label{eq:S_negative}
S_{\Omega,\omega}(\phi_g)=\inf_{\phi\in X} S_{\Omega,\omega}(\phi)=\inf_{\phi\in\mathcal{M}} S_{\Omega,\omega}(\phi)<0,
\end{align}
with $\mathcal{M}$ the Nehari manifold. Moreover, there exists some constant $M>0$ depending on $d, p, V, \Omega, \omega$ such that
\begin{equation}\label{lem:uniform_bound}
	\| \phi \|_X \leq M,\quad \forall \phi \in \{\phi\in X\ :\ S_{\Omega, \omega} (\phi) \leq 0\}.
\end{equation}
\end{proposition}

\begin{remark}
The assumption for the potential function $V$ in \cref{thm:exist} is a bit more general than that in Ref. \refcite{LYZ}, where the proof remains essentially the same with subtle modifications. Such generalized assumption allows us to cover more practical cases such as the harmonic-plus-lattice potential of the form $V(\mathbf{x})=\sum_{j=1}^d\big(\frac12 \gamma_j^2x_j^2+\kappa\sin^2(\pi x_j/2)\big)$ and the harmonic-plus-quartic potential of the form $V(\mathbf{x})=\frac{c_1}{4}(|\mathbf{x}|^2-c_2)^2$, where $\kappa,c_1,c_2>0$.
\end{remark}

In the rest of this paper, we denote the action ground states as $\phi_{\Omega, \omega}$ to more clearly show its dependence on $\omega$ and $\Omega$.
For the non-rotating case (i.e., $\Omega=0$), we adopt the following simplified notations by dropping the $\Omega$ dependence
\begin{equation}\label{eq:simplified_not}
S_{\omega}=S_{0,\omega},\quad  K_{\omega} =K_{0,\omega},\quad  \phi_{\omega} = \phi_{0, \omega}, \quad \mathcal{A}_{\omega} = \mathcal{A}_{0, \omega}, \quad  \mathcal{B}_m = \mathcal{B}_{0, m}.
\end{equation}
Also, we abbreviate $\lambda_0(\Omega)$ with $\Omega = 0$ as $\omega_0$:
\begin{equation}\label{eq:omega_0_def}
    \omega_0  := \lambda_0(0) = \inf_{0 \neq \phi \in X} \frac{ \int_{\mathbb{R}^d} \left( \frac12|\nabla \phi|^2 + V|\phi|^2 \right) \mathrm{d} \mathbf{x}}{\| \phi \|_{L^2}^2} > 0.
\end{equation}

\section{Non-rotating case}\label{sec:3}
In this section, we consider the non-rotating case (i.e., $ \Omega = 0$) and prove a complete equivalence between the action ground states $\phi_\omega$ and the energy ground states $\phi^E_m$:
\begin{equation*}
    \phi_\omega\in \mathcal{A}_\omega = \argmin_{\phi \in X} S_{\omega}(\phi), \qquad \phi_{m}^E \in \mathcal{B}_m=\argmin_{\substack{\phi \in X\\ \|\phi\|_{L^2}^2=m>0}} E(\phi).
\end{equation*}
Simplified notations in \eqref{eq:simplified_not} and \eqref{eq:omega_0_def} are adopted throughout this section.

\subsection{Action ground states are energy ground states}
    \begin{theorem}\label{thm:actiontoenergy}
        For any $ \omega < - \omega_0 $, an action ground state $ \phi_\omega $ associated with $\omega$ is also an energy ground state with mass $ m = \| \phi_\omega \|_{L^2}^2 $.
    \end{theorem}

	\begin{proof}
		By \cref{thm:exist}, $\phi_\omega$ is a global minimizer of $S_{\omega}$ in $X$. We use proof by contradiction. Suppose that there exists $ \tilde{\phi} \in X$ such that $\| \tilde{\phi}\|_{L^2}^2 = m $ and $E(\tilde{\phi}) < E(\phi_\omega)$, then $S_\omega(\tilde{\phi}) < S_\omega(\phi_\omega)$, and this leads to a contraction, which  concludes the proof.
	\end{proof}

    \subsection{Energy ground states are action ground states}\label{sec:3.2}
        \begin{theorem}\label{thm:energytoaction}
        For any $ m > 0 $, an energy ground state $ \phi^E_m $ with mass $m$ is also an action ground state associated with $ \omega = -\mu(\phi^E_m) $ where $\mu$ is defined in \eqref{eq:mu_def} with $\Omega = 0$.
    \end{theorem}
    The proof will be given with the help of the following lemmas.
    \begin{lemma}[Uniqueness]\label{lem:uniqueness}
		Let $ \omega < - \omega_0 $. For any $\phi^1_\omega, \phi^2_\omega \in \mathcal{A}_\omega$, we have
		\begin{equation}\label{eq:unique_action}
			| \phi^1_\omega(\vx) | = |\phi^2_\omega(\vx)| \quad \text{\rm a.e. } \vx \in \R^d.
		\end{equation}
		Let $ m > 0 $. For any $\phi^{E,1}_m, \phi^{E, 2}_m \in \mathcal{B}_m$, we have
		\begin{equation}\label{eq:unique_energy}
			| \phi^{E,1}_m(\vx) | = |\phi^{E,2}_m(\vx)| \quad \text{\rm a.e. } \vx \in \R^d.
		\end{equation}
    \end{lemma}

    \begin{proof}
	By the inequality (see Ref.~\refcite{Lieb})
 	\begin{equation}\label{eq:gradient}
            \int_{\mathbb{R}^d}|\nabla|\phi||^2\rmd\vx \leq \int_{\mathbb{R}^d}|\nabla\phi|^2 \rmd\vx, \quad \phi \in H^1(\mathbb{R}^d),
 	\end{equation}
 	recalling \eqref{action}, we get
 	\begin{equation}
 		S_{\omega}(|\phi|)\leq S_{\omega}(\phi), \qquad \phi \in X,
 	\end{equation}
 	which implies that, if $\phi$ is an action ground state, $|\phi|$ is also an action ground state. Then \eqref{eq:unique_action} holds if we have the uniqueness of the nonnegative action ground state or the density.
 		
    Consider the density $ \rho(\vx) = |\phi(\vx)|^2 \geq 0 $ with $\sqrt{\rho} \in X$. Following the process in Ref.~\refcite[Lemma~A.1]{Lieb2000}, we have $S_{\omega}(\sqrt{\rho}) $ is strictly convex with respect to $\rho$, which implies the uniqueness of the density of action ground states and proves \eqref{eq:unique_action}. The proof of \eqref{eq:unique_energy} is similar and can be found in Refs.~\refcite{BaoCai,Lieb2000}.
    \end{proof}

\begin{remark}
In \cref{lem:uniqueness}, if, in addition, $ V\in L^2_{\mathrm{loc}} $, then $\phi_\omega$ and $\phi^E_m$ can be chosen as strictly positive, and the action and energy ground states are unique up to a constant phase shift (see Refs.~\refcite{Lieb,Lieb2000}).
\end{remark}

    The proof of \cref{thm:energytoaction} will be done by studying the properties of the mass of action ground states as a function of $\omega$, which is well-defined due to \cref{lem:uniqueness}. Define the action ground state mass $M_g:(-\infty, -\omega_0) \rightarrow \R^+$ as
    \begin{equation}\label{eq:Mg_def}
        M_g(\omega) := \| \phi_\omega \|_{L^2}^2, \quad \omega < -\omega_0.
    \end{equation}
    In the proof of the following \cref{lem1,lem2,lem3}, we assume that $ \phi_\omega $ is the unique nonnegative action ground state associated with $ \omega < - \omega_0 $ and $ \phi^E_m $ is the unique nonnegative energy ground state associated with $ m > 0 $.

    \begin{lemma}\label{lem1}
        The function $ M_g $ is injective on $(-\infty, -\omega_0)$, i.e., for any $ \omega_1, \omega_2 \in (-\infty, - \omega_0) $, $ M_g(\omega_1) = M_g(\omega_2) $ implies $ \omega_1 = \omega_2 $.
    \end{lemma}

    \begin{proof}
        Let $ m_1 = M_g(\omega_1) $ and $ m_2 = M_g(\omega_2) $. By \cref{thm:actiontoenergy}, we have $ \phi_{\omega_1} = \phi^E_{m_1} $ and $ \phi_{\omega_2} = \phi^E_{m_2} $. If $ m_1 = m_2 $, by the uniqueness of non-negative energy ground states in \cref{lem:uniqueness}, we have $ \phi^E_{m_1} = \phi^E_{m_2} $ and thus $ \phi_{\omega_1} = \phi_{\omega_2} =: \phi_\ast $. Then $ \phi_\ast $ will solve the stationary NLS
	\begin{equation}
		- \frac12\Delta \phi_\ast + V \phi_\ast + \omega \phi_\ast + |\phi_\ast|^{p-1}\phi_\ast = 0, \quad \omega = \omega_1, \omega_2,
	\end{equation}
	which implies $ \omega_1 = \omega_2 $ and completes the proof.
    \end{proof}
	
    \begin{lemma}\label{lem2}
        The function $ M_g(\omega) := \| \phi_\omega \|_{L^2}^2 $ is continuous on $ (-\infty, -\omega_0) $.
    \end{lemma}

    \begin{proof}
        We first show that $S_{{\omega_\ast}}(\phi_{\omega})\to S_{\omega_\ast}(\phi_{\omega_\ast})$ as $\omega \rightarrow \omega_\ast$ for any fixed $\omega_\ast \in (-\infty, -\omega_0)$. Recalling the variational characterization of the action ground states in \eqref{eq:S_negative} and the expression of the action functional \eqref{action}, we obtain
	\begin{align}
		&S_{\omega_\ast}(\phi_{\omega_\ast})<S_{\omega_\ast}(\phi_{\omega})<S_{\omega}(\phi_{\omega})<S_{\omega}(\phi_{\omega_\ast}), \qquad \omega_\ast<\omega<-\omega_0, \label{eq:Somega_1}\\
		&S_{\omega}(\phi_{\omega})<S_{\omega}(\phi_{\omega_\ast})<S_{\omega_\ast}(\phi_{\omega_\ast})<S_{\omega_\ast}(\phi_{\omega}), \qquad \omega<\omega_\ast. \label{eq:Somega_2}
	\end{align}
	Here, the equality case $S_{\omega_\ast}(\phi_{\omega_\ast}) = S_{\omega_\ast}(\phi_{\omega})$ is excluded since otherwise $\phi_\omega$ is a minimizer of both $S_\omega$ and $S_{\omega_\ast}$, and thus satisfies \cref{model0} for both $\omega$ and $\omega_\ast$, which is impossible. The same argument excludes $S_{\omega}(\phi_{\omega}) = S_{\omega}(\phi_{\omega_\ast})$. Then for any $\omega_\ast<\omega<-\omega_0$, by \eqref{eq:Somega_1}, we have
	\begin{align}\label{eq:Somg-est1}
		|S_{\omega_\ast}(\phi_{\omega})-S_{\omega_\ast}(\phi_{\omega_\ast})|<|S_{\omega}(\phi_{\omega_\ast})-S_{\omega_\ast}(\phi_{\omega_\ast})|=|\omega-\omega_\ast|M_g(\omega_\ast).
	\end{align}
	Then, for $\omega < \omega_\ast$, by \eqref{eq:Somega_2}, we have
	\begin{equation}\label{eq:Somg-est2}
		|S_{\omega_\ast}(\phi_{\omega})-S_{\omega_\ast}(\phi_{\omega_\ast})|<|S_{\omega_\ast}(\phi_{\omega})-S_{\omega}(\phi_{\omega})|=|\omega-\omega_\ast|M_g(\omega).
	\end{equation}
	For any fixed $\delta > 0$, when $\omega_\ast-\delta<\omega<\omega_\ast$, by \cref{action} and \cref{eq:S_negative}, we have
	\begin{equation}\label{eq:delta}
		S_{\omega_\ast-\delta}(\phi_{\omega})<S_{\omega}(\phi_{\omega})<0,
	\end{equation}
	which implies, by \eqref{lem:uniform_bound},
	\begin{equation}\label{eq:M_bound}
		M_g(\omega) \leq C, \quad \omega_\ast-\delta<\omega<\omega_\ast,
	\end{equation}
	where $C$ is some constant depending only on $\omega_\ast,\delta,d,p,V$. Combining \eqref{eq:Somg-est1}, \eqref{eq:Somg-est2} and \eqref{eq:M_bound} yields that $S_{\omega_\ast}(\phi_{\omega})\to S_{\omega_\ast}(\phi_{\omega_\ast})$ as $\omega\to\omega_\ast$.

        As a consequence, for any sequence $\{\omega^k\}_{k=1}^{\infty}\subset(-\infty,\omega_0)$ converging to $\omega_\ast$, $\{\phi_{\omega^k}\}_{k=1}^{\infty}$ minimizes the functional $S_{\omega_\ast}$ in $X$. Hence, by \eqref{eq:S_negative} and \eqref{lem:uniform_bound}, $\{\phi_{\omega^k}\}_{k=1}^{\infty}$ is uniformly bounded in $X$. Utilizing the same argument used in the proof of Theorem 3.1 in Ref. \refcite{LYZ}, we can easily apply the compact embedding given in Lemma~\ref{lem:Xembed} and the weak lower-semicontinuity of the norm $\|\cdot\|_{X}$ to conclude that $\{\phi_{\omega^k}\}_{k=1}^{\infty}$ contains a subsequence converging strongly to the unique nonnegative action ground state $\phi_{\omega_\ast}$ in $X$ (and therefore in $L^2(\mathbb{R}^d)$). The uniqueness of this limit leads to $M_g(\omega)=\|\phi_{\omega}\|^2_{L^2}\to\|\phi_{\omega_\ast}\|^2_{L^2}=M_g(\omega_\ast)$ as $\omega\to\omega_\ast$. The proof is completed.
    \end{proof}

    \begin{proposition}[Limits of mass in $\omega$]\label{lem3}
	$ M_g(\omega) \rightarrow 0 $ as $ \omega \rightarrow {-\omega_0}^{-} $ and $ M_g(\omega)  \rightarrow \infty $ as $ \omega \rightarrow - \infty $.
    \end{proposition}
    \begin{proof}
    We first show that $ M_g(\omega) \rightarrow 0 $ as $ \omega \rightarrow {-\omega_0}^{-} $. {By \cref{eq:S_negative,action,lem:uniform_bound}, similar to \cref{eq:delta},} we have $ \phi_\omega $ is uniformly bounded in $ X $ for $ \omega \in (-\omega_0 -\delta, -\omega_0) $ with any fixed $\delta > 0$. Recalling \cref{action,eq:omega_0_def}, we have
    \begin{equation}\label{S}
	S_\omega(\phi_\omega) \geq (\omega + \omega_0) \| \phi_\omega \|_{L^2}^2 + \frac{2}{p+1} \| \phi_\omega \|_{L^{p+1}}^{p+1} \geq (\omega + \omega_0) \| \phi_\omega \|_{L^2}^2.
    \end{equation}
    {From \cref{S}, recalling \cref{eq:K_def} and $K_\omega(\phi_\omega) = 0$, and $S_\omega(\phi_\omega)<0$ in \cref{eq:S_negative}}, we obtain
    \begin{equation}
	0 > S_\omega(\phi_\omega) = -\frac{p-1}{p+1}\| \phi_\omega \|_{L^{p+1}}^{p+1} \geq (\omega + \omega_0) \| \phi_\omega \|_{L^2}^2 \rightarrow 0\quad  \mbox{as}\quad \omega \rightarrow {-\omega_0}^{-}.
    \end{equation}
    It follows that $  \| \phi_\omega \|_{L^{p+1}}^{p+1} \rightarrow 0 $. To control $\| \phi_\omega \|_{L^2}$ by $\| \phi_\omega \|_{L^{p+1}}^{p+1}$, we claim that for any $A > 0$, there exists $r_0 > 0$ such that for $r > r_0$, we have
    \begin{equation}\label{controlL2byLp}
        \| \phi \|_{L^2}^2 \leq  |B_r|^\frac{p-1}{p+1} \| \phi \|_{L^{p+1}}^2 + \frac{1}{A} \| \phi \|_X^2, \quad \phi \in X,
    \end{equation}
    where $B_r=\{\vx\in\mathbb{R}^d: |\vx|<r\}$. In fact, by the confining condition $\lim_{|\vx| \rightarrow \infty} V(\vx) = +\infty$, for any $A > 0$, there exists $r_0>0$ such that $V(\vx) + 1 \geq A$ for any $ |\vx| \geq r_0 $. By H\"older's inequality, for any $r > r_0$, we have
    \begin{align*}
	\| \phi \|_{L^2}^2
	= \int_{B_r} |\phi|^2 \rmd \vx + \int_{|\vx| \geq r} |\phi|^2 \rmd \vx
        &\leq |B_r|^\frac{p-1}{p+1} \| \phi \|_{L^{p+1}}^2 + \frac{1}{A} \int_{|\vx| \geq r} (V(\vx) + 1) |\phi|^2 \rmd \vx \\
        &\leq |B_r|^\frac{p-1}{p+1} \| \phi \|_{L^{p+1}}^2 + \frac{1}{A} \| \phi \|_X^2,
    \end{align*}
    which proves the claim \eqref{controlL2byLp}. Since $ \phi_\omega $ is uniformly bounded in $ X $, one has $ \| \phi_\omega \|_{L^2}^2 \rightarrow 0 $ by \eqref{controlL2byLp}.

		Then we shall show that $ M_g(\omega) \rightarrow \infty $ as $ \omega \rightarrow - \infty $. Note that for all $ \phi \in X $,
		\begin{equation*}
			-\| \phi_\omega \|_{L^2}^2 \leq \frac{S_\omega(\phi_\omega)}{|\omega|} \leq \frac{S_\omega(\phi)}{|\omega|}
           = \frac{E(\phi)}{|\omega|} - \| \phi \|_{L^2}^2.
		\end{equation*}
  Namely, $M_g(\omega)\geq \| \phi \|_{L^2}^2-\frac{E(\phi)}{|\omega|}$, $\forall \phi \in X$. The arbitrariness of $\phi\in X$ implies that $ M_g(\omega) \rightarrow \infty $ as $ \omega \rightarrow -\infty $.
	\end{proof}

	\begin{proof}[Proof of \cref{thm:energytoaction}]
        By \cref{lem1,lem2}, we have $ M_g(\omega) $ is a strictly monotone function on $(-\infty, -\omega_0)$. Furthermore, by \cref{lem3}, $M_g(\omega)$ is strictly decreasing and is surjective from $(-\infty, -\omega_0)$ to $(0, \infty)$. Hence, for any $m>0$, there exists a unique $\omega$ satisfying $m=M_g(\omega)$.
		Let $\phi_\omega$ be an action ground state associated with $\omega$. By  \cref{thm:actiontoenergy}, we have
		\begin{equation}
			E(\phi_\omega) = E(\phi^E_m), \quad \| \phi_\omega \|_{L^2}^2 = \| \phi^E_m \|_{L^2}^2 = m,
		\end{equation}
		which implies
		\begin{equation}
			S_\omega(\phi^E_m) = E(\phi^E_m) + \omega \| \phi^E_m \|_{L^2}^2 = S_\omega(\phi_\omega),
		\end{equation}
		which further implies that $\phi^E_m$ is an action ground state associated with $\omega$. That $\omega = -\mu(\phi^E_m)$ follows from noting that $\phi^E_m$ solves \eqref{model0}.
	\end{proof}

	\begin{remark}
            By \cref{eq:Somega_1,eq:Somega_2}, and the analysis in the proof of \cref{lem2}, we also obtain that $ S_g(\omega):=S_\omega(\phi_\omega) $ is continuous and strictly increasing with respect to $ \omega $ on $ (-\infty, -\omega_0) $. Moreover, the proof of \cref{lem3} implies that
            \begin{equation*}
                \lim_{\omega \rightarrow -\infty}S_g(\omega)=-\infty\quad \mbox{and}\quad  \lim_{\omega\rightarrow -\omega_0^-} S_g(\omega)=0.
            \end{equation*}
            In fact, the asymptotics of the limits of mass $ M_g(\omega)$ and action $S_g(\omega)$ in  $\omega$ can be derived as well. These will be discussed in the next section that covers the rotating case.
	\end{remark}

The established \cref{thm:actiontoenergy,thm:energytoaction} together show the complete equivalence of the action ground state and the energy ground state for the non-rotating defocusing NLS. Consequently, mathematical studies for the ground state, such as its stability as a soliton, can be explored in one of the setups.

\begin{remark}[Comparison with the focusing nonlinearity]
	When there is a focusing nonlinearity in the NLS such as the focusing NLS considered in Refs.~\refcite{Dovetta,LYZ} and the focusing-defocusing NLS considered in Refs.~\refcite{Lewin,Carles,Jeanjean2021}, an interesting observation is that the role of the action ground state and the energy ground state seems to exchange. To be precise, it is proved in Refs.~\refcite{Dovetta,Jeanjean2021} that any energy ground state is also an action ground state. While, whether an action ground state is also an energy ground state is more complicated and is in general false. These should be compared with \cref{thm:actiontoenergy,thm:energytoaction} and \cref{thm:actiontoenergy_new,thm:Non-equivalence} in the next section.
\end{remark}

\section{Rotating case}\label{sec:rot}
We move on to consider the theoretical studies in the rotating case, i.e., $0<\Omega<\Omega_\text{max}$. Similar to the non-rotating case, we first focus on the relationship between the action ground state $\phi_{\Omega, \omega}$ and the energy ground state $\phi^E_m$:
\begin{equation*}
	\phi_{\Omega,\omega} \in \mathcal{A}_{\Omega, \omega} = \argmin_{\phi \in X} S_{\Omega, \omega}(\phi), \quad \phi_{m}^E \in \mathcal{B}_{\Omega, m} = \argmin_{\substack{\phi \in X\\ \|\phi\|_{L^2}^2=m>0}} E(\phi).
\end{equation*}
Then we study the asymptotic properties of the action ground states with respect to the parameters $\omega$ and $\Omega$. In this section, we use $\phi_{\Omega, \omega}$ to denote an arbitrary element in $\mathcal{A}_{\Omega, \omega}$.

\subsection{Relation with energy ground states}
In the case of rotation, there is generally no theoretical guarantee for the uniqueness of the action or energy ground states, which becomes the primary difference and difficulty in this context. 
Here, we choose to follow the work of Ref.~\refcite{Dovetta}, allowing the possibility of multiple non-trivial action ground states of a given $\omega$. Such an argument will lead us to a clear characterization of equivalence/non-equivalence conditions in the end, which will guide the numerical investigations in the next section.

First of all, \cref{thm:actiontoenergy} can also be extended to the rotating case, i.e., an action ground state $\phi_{\Omega,\omega}$ is an energy ground state with mass $m=\|\phi_{\Omega,\omega}\|_{L^2}^2$. {More precisely, we have the following conditional equivalence of the energy ground states and the action ground states.} 
\begin{theorem}[Conditional equivalence]\label{thm:actiontoenergy_new}
    For any $ \omega < - \lambda_0(\Omega) $, an action ground state $ \phi_{\Omega,\omega} $ is also an energy ground state with mass $ m = \| \phi_{\Omega,\omega} \|_{L^2}^2 $, and all the energy ground states $\phi_m^E$ with precisely this mass $ m $ are also action ground states associated with $\omega$. Moreover, we have
    \begin{equation}\label{eq:action_energy_rot}
        \mathcal{A}_{\Omega,\omega} = \bigcup_{m \in \mathcal{N}_{\Omega, \omega}} \mathcal{B}_{\Omega, m},
    \end{equation}
    where $\mathcal{N}_{\Omega,\omega} := \left\{\| \phi \|_{L^2}^2: \phi \in \mathcal{A}_{\Omega,\omega} \right\}$ is the set of mass of action ground states with given $\omega$ and $\Omega$.
\end{theorem}

\begin{proof}
    Following the proof of \cref{thm:actiontoenergy}, we have $ \phi_{\Omega,\omega} $ is also an energy ground state with mass $ m = \| \phi_{\Omega,\omega} \|_{L^2}^2 $, which implies $\mathcal{A}_{\Omega,\omega} \subset \bigcup_{m \in \mathcal{N}_{\Omega, \omega}} \mathcal{B}_{\Omega, m}$.  On the other hand, for this $m$, since $\phi_{\Omega, \omega}$ is an energy ground state with mass $m$, we have
    \begin{equation*}
        E(\phi) = E( \phi_{\Omega,\omega} ), \quad \| \phi \|_{L^2}^2 = m = \| \phi_{\Omega,\omega} \|_{L^2}^2, \quad \forall \phi \in \mathcal{B}_{\Omega, m},
    \end{equation*}
    which implies $ S_{\Omega,\omega}(\phi) = S_{\Omega,\omega}( \phi_{\Omega,\omega} ) $ for all $\phi \in \mathcal{B}_{\Omega, m}$. Thus any $\phi \in \mathcal{B}_{\Omega, \omega}$ is also an action ground state and it holds $ \mathcal{B}_{\Omega, m} \subset \mathcal{A}_{\Omega,\omega} $. The equation \cref{eq:action_energy_rot} follows from the above arguments immediately and the proof is completed.
\end{proof}

By \cref{thm:actiontoenergy_new}, if the mass of all action ground states as a set can cover $(0, \infty)$, i.e.,
\begin{equation}\label{eq:mass_all}
	\mathcal{N} :=\bigcup_{\omega \in (-\infty, -\lambda_0(\Omega))} \mathcal{N}_{\Omega, \omega} = (0, \infty),
\end{equation}
then any energy ground state $\phi^E_m$ with mass $m>0$ is also an action ground state with $\omega = - \mu(\phi^E_m)$ with $\mu$ defined in \cref{eq:mu_def}, and we obtain the complete equivalence of the two. If $\mathcal{N} \neq (0, \infty) $, then for any $ m \in (0, \infty) \setminus \mathcal{N}$, the energy ground state with mass $m$ is clearly not an action ground state. {However, due to the lack of uniqueness in the rotating case, we cannot follow the argument in the non-rotating case to prove \cref{eq:mass_all} and obtain the complete equivalence. In fact, it is unclear whether \cref{eq:mass_all} will hold.} In the following, instead of proving or disproving \cref{eq:mass_all}, we shall give a detailed characterization of the structure of the set $\mathcal{N}$, which allows us to derive more precise equivalence and non-equivalence conditions. Notably, such conditions will guide our numerical experiments to find examples of non-equivalence (see \cref{sec:num_non-eq}).

    Similar to the function $M_g$ defined in \eqref{eq:Mg_def}, we now define
    \begin{equation}\label{eq:m_def}
        m^\ast(\omega) := \sup\mathcal{N}_{\Omega,\omega}, \qquad m_\ast(\omega) := \inf\mathcal{N}_{\Omega,\omega}.
    \end{equation}
    As in \cref{sec:3.2}, we shall study the structure of $\mathcal{N}$ through the properties of two mass functions $m^\ast$ and $m_\ast$, including continuity, monotonicity and the limit.
    For ease of presentation, we define another inner product $ \lrang{\cdot}{\cdot}_{R(\kappa)} $ on $ X $ as
    \begin{equation}\label{eq:inner_R_def}
		\lrang{\cdot}{\cdot}_{R(\kappa)} = \lrang{R \cdot}{\cdot} + \kappa \lrang{\cdot}{\cdot}, \qquad
   R= - \frac12\Delta + V - \Omega L_z, 
    \end{equation}
    with $ \kappa > \max\{-\lambda_0(\Omega), 0\} $ a fixed constant. Then, under Assumption \ref{lem:S-welldef}, $ \| \cdot \|_{R(\kappa)} = \sqrt{Q_{\Omega, \kappa}(\cdot)} $ is an equivalent norm on $ X $.

    \begin{lemma}\label{prop:aux_convergence}
		Let $\omega_\ast \in (-\infty, -\lambda_0(\Omega))$. For any $ \omega^k \rightarrow \omega_\ast^+ $ or $ \omega^k \rightarrow \omega_\ast^- $, there exist $ \phi \in \mathcal{A}_{\Omega,\omega_\ast} $ and a subsequence $ \{ \omega^{n_k} \}_k $ such that
        \begin{equation*}
	   	\phi_{\Omega, \omega^{n_k}} \rightharpoonup \phi \text{ in } X,\quad \text{ as } k \rightarrow \infty,
        \end{equation*}
        where $\phi_{\Omega, \omega^{n_k}}$ is the action ground state associated with $\omega^{n_k}$ and $\Omega$.
    \end{lemma}

    \begin{proof}
		We start with the case $ \omega^k \rightarrow \omega_\ast^+ $. Note that for any $ \omega $ satisfying $ \omega_\ast < \omega < -\lambda_0(\Omega) $, one has, similar to \eqref{eq:Somega_1},
		\begin{equation}\label{eq:relation_right}
			S_{\Omega, \omega_\ast}(\phi_{\Omega, \omega_\ast}) < S_{\Omega, \omega_\ast}(\phi_{\Omega, \omega}) < S_{\Omega, \omega}(\phi_{\Omega, \omega}) < S_{\Omega, \omega}(\phi_{\Omega, \omega_\ast}).
		\end{equation}
		Recalling \eqref{action}, we have
		\begin{equation}
			S_{\Omega, \omega}(\phi_{\Omega, \omega_\ast}) \rightarrow S_{\Omega, \omega_\ast}(\phi_{\Omega, \omega_\ast}) \text{\quad  as \quad} \omega \rightarrow \omega_\ast^+,
		\end{equation}
		which, together with \eqref{eq:relation_right} yields
  \begin{subequations}
		\begin{align}
			&S_{\Omega, \omega_\ast}(\phi_{\Omega, \omega}) \rightarrow S_{\Omega, \omega_\ast}(\phi_{\Omega, \omega_\ast}), \label{eq:S_limit_1}\\
			&S_{\Omega, \omega}(\phi_{\Omega, \omega}) \rightarrow S_{\Omega, \omega_\ast}(\phi_{\Omega, \omega_\ast}) .\label{eq:S_limit_2}
		\end{align}
  \end{subequations}
		By \cref{eq:S_limit_2,eq:K_def}, we have
		\begin{equation*}
			\| \phi_{\Omega, \omega} \|_{L^{p+1}} \rightarrow \| \phi_{\Omega, \omega_\ast} \|_{L^{p+1}}, \text{\quad  as \quad} \omega \rightarrow \omega_\ast^+.
		\end{equation*}
		By \cref{eq:S_limit_1,eq:S_negative}, one has $ S_{\Omega, \omega_\ast}(\phi_{\Omega, \omega}) < 0 $ for $ \omega>\omega_\ast $ sufficiently close to $ \omega_\ast $, which implies, by \cref{lem:uniform_bound}, $ \phi_{\Omega, \omega} $ is uniformly bounded in $ X $ for $ \omega_\ast < \omega < \omega_\ast + \delta < -\lambda_0(\Omega)$ with some $ \delta > 0 $ sufficiently small. Then, for any sequence $ \omega^k \rightarrow \omega_\ast^+ $ satisfying $ \omega_\ast < \omega^{k} < \omega_\ast + \delta < -\lambda_0(\Omega)$, there exists a subsequence (still denoted by $ \{\omega^k\}_k $) and $ \phi \in X $ such that $ \phi_{\Omega, \omega^k} \rightharpoonup \phi$ weakly in $X$. It remains to show that $ \phi \in \mathcal{A}_{\Omega,\omega_\ast} $. By \cref{lem:Xembed}, $ \phi_{\Omega, \omega^k} \rightarrow \phi $ strongly in $ L^2 $ and $ L^{p+1} $. Recalling \cref{eq:inner_R_def}, one has
        \[ S_{\Omega, \omega_\ast}(\phi) = \| \phi \|_{R(\kappa)}^2 - (\kappa-\omega_\ast) \| \phi \|_{L^2}^2 + \frac{2}{p+1} \| \phi \|_{L^{p+1}}^{p+1}. \]
        Using the weak lower-semicontinuity of the norm $\| \cdot \|_{R(\kappa)}$ in $X$, the strong convergence of $\phi_{\Omega, \omega^k}$ in $L^2$ and $L^{p+1}$, and \cref{eq:S_limit_1}, we obtain
		\begin{align*}
			S_{\Omega, \omega_\ast}(\phi) \leq \liminf_{k \rightarrow \infty} S_{\Omega, \omega_\ast}(\phi_{\Omega, \omega^k}) = S_{\Omega, \omega_\ast}(\phi_{\Omega, \omega_\ast}).
		\end{align*}
		Hence, $S_{\Omega, \omega_\ast}(\phi)=S_{\Omega, \omega_\ast}(\phi_{\Omega, \omega_\ast})$ and $ \phi \in \mathcal{A}_{\Omega,\omega_\ast} $.
		
		Then we consider the case $ \omega^k \rightarrow \omega_\ast^{-} $. For any $ \omega $ satisfying $ \omega < \omega_\ast $, one has
		\begin{equation*}
			S_{\Omega, \omega}(\phi_{\Omega, \omega}) < S_{\Omega, \omega}(\phi_{\Omega, \omega_\ast}) < S_{\Omega, \omega_\ast}(\phi_{\Omega, \omega_\ast}) < S_{\Omega, \omega_\ast}(\phi_{\Omega, \omega}).
		\end{equation*}
		We claim that
		\begin{equation}\label{eq:S_limit_prove}
			 S_{\Omega, \omega_\ast}(\phi_{\Omega, \omega}) - S_{\Omega, \omega}(\phi_{\Omega, \omega}) \rightarrow 0 \quad \text{as } \omega \rightarrow \omega_\ast^-.
		\end{equation}
		To show \cref{eq:S_limit_prove}, it suffices to show that $ \| \phi_{\omega} \|_{L^2} $ is uniformly bounded when $\omega \rightarrow \omega_\ast^{-}$. Note that, for any $\delta > 0$,
		\begin{equation*}
			S_{\Omega, \omega_\ast-\delta}(\phi_{\Omega, \omega}) < S_{\Omega, \omega}(\phi_{\Omega, \omega}) < 0, \quad \omega_\ast-\delta < \omega <\omega_\ast,
		\end{equation*}
		which implies by \cref{lem:uniform_bound} that $ \{\phi_{\Omega, \omega}\}_{\omega_{\ast}-\delta < \omega < \omega_\ast} $ is uniformly bounded in $ X $ and thus the claim \cref{eq:S_limit_prove} is proved. 
		The rest of the proof follows a similar procedure as in the case $\omega \rightarrow \omega_\ast^+$ and we shall omit it for brevity.
    \end{proof}

    Now we are able to characterize the mass functions $m^\ast$ and $m_\ast$ as derivatives of the action function with respect to $\omega$.
    That is to define $S^g_{\Omega}: (-\infty, -\lambda_0(\Omega)) \rightarrow \mathbb{R} $ as
    \begin{equation}\label{eq:Sg_Omega_def}
		S^g_{\Omega}(\omega) := S_{\Omega,\omega}(\phi_{\Omega,\omega}), \quad \omega < - \lambda_0(\Omega),
    \end{equation}
    and we have the following result for its relation with the mass.
    \begin{proposition}\label{thm:differentiability of Sg}
        The left and right derivative of $ S^g_{\Omega} $ exist everywhere on $ (- \infty, -\lambda_0(\Omega)) $, and we denote them by
        ${S^g_{\Omega}}'_{-},{S^g_{\Omega}}'_{+}$ respectively. Then, we have
        \begin{equation*}
            {S^g_{\Omega}}'_{-}(\omega) = m^\ast(\omega), \quad {S^g_{\Omega}}'_{+}(\omega) = m_\ast(\omega), \quad \omega < -\lambda_0(\Omega).
		\end{equation*}
    \end{proposition}
	
    \begin{proof}
		Let $ f_\phi(\omega) = S_{\Omega,\omega}(\phi) $. Then $ f_\phi(\cdot) $ is a linear function on $ (-\infty, -\lambda_0(\Omega)) $ for each given $ \phi \in X $ and thus is concave. Since
		\begin{equation*}
		  S^g_{\Omega}(\omega) = \inf_{\phi \in X} S_{\Omega,\omega}(\phi) = \inf_{\phi \in X} f_\phi(\omega)
		\end{equation*}
		is the infimum of a family of concave functions, $S_\Omega^g$ is also a concave function on $ (-\infty, -\lambda_0(\Omega)) $ and the one-side derivatives exist everywhere.
		
		For any $\omega_\ast < -\lambda_0(\Omega) $, we first show that ${S^g_{\Omega}}'_{+}(\omega_\ast) = m_\ast(\omega_\ast)$. Let $ \phi_{\Omega, \omega_\ast} \in \mathcal{A}_{\Omega, \omega_\ast} $ be an action ground state associated with $\omega_\ast$.
		For any $ \omega $ satisfying $ \omega_\ast < \omega < -\lambda_0(\Omega)$, by \cref{eq:relation_right}, we obtain 
		\begin{align*}
			{S^g_{\Omega}}'_{+}(\omega_\ast)
			&= \lim_{\omega \rightarrow \omega_\ast^+} \frac{S_{\Omega,\omega}(\phi_{\Omega,\omega}) - S_{\Omega,\omega_\ast}(\phi_{\Omega,\omega_\ast})}{\omega - \omega_\ast} \notag \\
			&\leq \lim_{\omega \rightarrow \omega_\ast^+} \frac{S_{\Omega,\omega}(\phi_{\Omega,\omega_\ast}) - S_{\Omega,\omega_\ast}(\phi_{\Omega,\omega_\ast})}{\omega - \omega_\ast} = \| \phi_{\Omega,\omega_\ast} \|_{L^2}^2.
		\end{align*}
		By the arbitrariness of $ \phi_{\Omega,\omega_\ast} \in \mathcal{A}_{\Omega, \omega_\ast} $, one has
		\begin{equation}\label{eq:leq}
			{S^g_{\Omega}}'_{+}(\omega_\ast) \leq \inf_{\phi \in \mathcal{A}_{\Omega, \omega_\ast}} \| \phi \|_{L^2}^2 = m_\ast(\omega_\ast).
		\end{equation}
		To show the converse, we note from \cref{eq:relation_right} that
		\begin{align*}
			{S^g_{\Omega}}'_{+}(\omega_\ast)
			&= \lim_{\omega \rightarrow \omega_\ast^+} \frac{S_{\Omega,\omega}(\phi_{\Omega,\omega}) - S_{\Omega,\omega_\ast}(\phi_{\Omega,\omega_\ast})}{\omega - \omega_\ast} \notag \\
			&\geq \limsup_{\omega \rightarrow \omega_\ast^+} \frac{S_{\Omega,\omega}(\phi_{\Omega,\omega}) - S_{\Omega,\omega_\ast}(\phi_{\Omega,\omega})}{\omega - \omega_\ast} = \limsup_{\omega \rightarrow \omega_\ast^+} \| \phi_{\Omega,\omega} \|_{L^2}^2.
		\end{align*}
		By \cref{prop:aux_convergence,lem:Xembed}, one can find $ \phi \in \mathcal{A}_{\Omega,\omega_\ast} $ and a sequence $ \{\omega^k\}_k \subset (\omega_\ast, -\lambda_0(\Omega)) $ satisfying $ \omega^k \rightarrow \omega_\ast^+ $ such that $\phi_{\Omega, \omega^k} \rightarrow \phi$ in $L^2$. Since  $ \phi \in \mathcal{A}_{\Omega,\omega_\ast} $, $ \| \phi \|_{L^2}^2 \geq m_\ast(\omega_\omega) $, which implies
		\begin{equation}\label{eq:geq}
			{S^g_{\Omega}}'_{+}(\omega_\ast) \geq \limsup_{\omega \rightarrow \omega_\ast^+} \| \phi_{\Omega,\omega} \|_{L^2}^2 \geq \| \phi \|_{L^2}^2 \geq m_\ast(\omega_\ast).
		\end{equation}
		Combining \cref{eq:leq,eq:geq}, we have
		\begin{equation*}
			\| \phi \|_{L^2}^2 = {S^g_{\Omega}}'_{+}(\omega_\ast) = m_\ast(\omega_\ast).
		\end{equation*}
		
		The proof of ${S^g_{\Omega}}'_{-}(\omega) = m^\ast(\omega)$ can be obtained similarly, and we omit the details for brevity.
    \end{proof}

    \begin{remark}\label{rem:limit_mass}
        As a by-product of the proof of \cref{thm:differentiability of Sg}, for $\phi$ in \cref{prop:aux_convergence}, we have $ \| \phi \|_{L^2}^2 = m_\ast(\omega_\ast) $ if $ \omega^k \rightarrow \omega_\ast^+ $ and $ \| \phi \|_{L^2}^2 = m^\ast(\omega_\ast) $ if $ \omega^k \rightarrow \omega_\ast^- $. Consequently,
        \[ m^\ast(\omega) = \sup \mathcal{N}_{\Omega, \omega} \in \mathcal{N}_{\Omega, \omega},\quad  m_\ast(\omega) = \inf \mathcal{N}_{\Omega, \omega} \in \mathcal{N}_{\Omega, \omega},\quad \forall \omega < -\lambda_0(\Omega). \]
    \end{remark}

  	{By \cref{thm:differentiability of Sg}, we can obtain more mathematical properties of $m_\ast$ and $m^\ast$.}
    \begin{corollary}\label{prop:mast}
		For $ m_\ast $ and $ m^\ast $ defined in \cref{eq:m_def}, we have
        \begin{enumerate}[label=(\roman*)]
		  \item $ m_\ast(\omega) $ and $ m^\ast(\omega) $ are strictly decreasing on $ (- \infty, -\lambda_0(\Omega) ) $;
		  \item $ m_\ast(\omega) \leq m^\ast(\omega) $ for any $ \omega \in (-\infty, -\lambda_0(\Omega)) $;
		  \item $ m_\ast(\omega_1) > m^\ast(\omega_2) $ if $ \omega_1 < \omega_2 $.
        \end{enumerate}
    \end{corollary}
	
    \begin{proof}
		The fact $m_\ast(\omega) \leq m^\ast(\omega)$ follows directly from the definition of the two functions. By \cref{thm:differentiability of Sg}, $ m^\ast $ and $ m_\ast $ are the left and right derivatives of the concave function $S^g_\Omega(\omega)$. Hence, they are decreasing and
        \begin{align}
            &m_\ast(\omega_1) \geq m^\ast(\omega_2), \quad \omega_1 < \omega_2 < - \lambda_0(\Omega). \label{eq:iii}
        \end{align}

        It remains to show that $ m_\ast $ and $ m^\ast $ are strictly decreasing and the equality in \cref{eq:iii} cannot be attained. If $ m_\ast(\omega_1) $ = $ m_\ast(\omega_2) $ or $ m^\ast(\omega_1) $ = $ m^\ast(\omega_2) $ or $ m_\ast(\omega_1) = m^\ast(\omega_2) $ for some $ \omega_1 < \omega_2 $, then there exists a constant $ m>0 $ such that
		\begin{equation*}
			m^\ast(\omega) = m_\ast(\omega) = m, \quad \forall \omega \in (\omega_1, \omega_2).
		\end{equation*}
        By \cref{thm:actiontoenergy_new}, $\phi_{\Omega,\omega} \in \mathcal{B}_{\Omega, m}$ for all $ \omega \in (\omega_1, \omega_2) $, which contradicts \cref{eq:action_energy_rot} since $\{\mathcal{A}_{\Omega, \omega}\}_{\omega<-\lambda_0(\Omega)}$ is certainly pairwise disjoint. Thus, the proof is completed.
    \end{proof}

    By \cref{prop:mast}, $ m_\ast $ and $ m^\ast $ are both monotonically decreasing and thus are continuous except a countable set of points. In fact, $m^\ast(\omega)$ and $m_\ast(\omega)$ share the same discontinuous points. At the continuous point $\omega$ of $m^\ast$ and $m_\ast$, we have $ S^g_{\Omega} $ is differentiable and $\mathcal{N}_{\Omega,\omega}$ is a singleton.

	\begin{proposition}\label{prop:leftrightcontinuous}
		$ m^\ast $ is left continuous and $ m_\ast $ is right continuous, i.e.,
		\begin{equation*}
			\lim_{\omega \rightarrow \omega_\ast^-} m^\ast(\omega) = m^\ast(\omega_\ast), \qquad \lim_{\omega \rightarrow \omega_\ast^+} m_\ast(\omega) = m_\ast(\omega_\ast),\qquad  \omega_\ast \in (-\infty, \lambda_0(\Omega)).
		\end{equation*}
	\end{proposition}

	\begin{proof}
            The result is a direct consequence of \cref{prop:aux_convergence,rem:limit_mass}. We briefly present the proof of the left continuity of $ m^\ast $. The right continuity of $ m_\ast $ is similar. By \cref{rem:limit_mass}, for any $\omega < -\lambda_0(\Omega)$, there exists $\phi_{\Omega, \omega} \in \mathcal{A}_{\Omega, \omega}$ such that $ m^\ast(\omega)  = \| \phi_{\Omega, \omega} \|_{L^2}^2 $.
            Then by \cref{prop:aux_convergence,rem:limit_mass} as well as the compact embedding of \cref{lem:Xembed}, there exists a sequence $\omega^k \rightarrow \omega_\ast^-$ such that
            \begin{equation*}
                \lim_{\omega \rightarrow \omega_\ast^-} m^\ast(\omega) = \lim_{k} m^\ast(\omega^k) = \lim_{k} \| \phi_{\Omega, \omega^k} \|_{L^2}^2 = m^\ast(\omega_\ast),
            \end{equation*}
            which yields the left continuity of $\omega^\ast$.
	\end{proof}

	\begin{proposition}[Limits of mass]\label{limit_of_mass}
		$ m^\ast(\omega) \rightarrow 0 $ as $ \omega \rightarrow {-\lambda_0(\Omega)}^{-} $ and $ m_\ast(\omega) \rightarrow \infty $ as $ \omega \rightarrow - \infty $.
	\end{proposition}
 \begin{proof}
The proof is very similar to that for \cref{lem3} and is omitted here for brevity.
 \end{proof}

        Denote $ \mathcal{I}(\omega) := [m_\ast(\omega), m^\ast(\omega)] $ for $\omega < -\lambda_0(\Omega)$. Then, the combination of the above properties of  $m_\ast,m^\ast$ directly leads to the following corollary.
	\begin{corollary}\label{cor:property_I}
             The collection of closed intervals $\{\mathcal{I}(\omega)\}_{\omega<-\lambda_0(\Omega)}$ is pairwise disjoint, i.e., $ \mathcal{I}(\omega_1) \cap \mathcal{I}(\omega_2) = \emptyset $ if $ \omega_1 \neq \omega_2 $. Moreover,
		\begin{equation*}
			\bigcup_{\omega \in (-\infty, -\lambda_0(\Omega))} \mathcal{I}(\omega) = (0,\infty).
		\end{equation*}
	\end{corollary}
With the disjoint closed interval $\mathcal{I}(\omega)$, we can now state the precise mathematical description for the case when energy ground states fail to be action ground states as follows.
        \begin{theorem}[Non-equivalence]\label{thm:Non-equivalence}
            If there exists some $ \tilde{\omega}\in(-\infty,-\lambda_0(\Omega)) $ such that $ \mathcal{I}(\tilde{\omega}) \setminus \mathcal{N}_{\Omega,\tilde{\omega}} \neq \emptyset $, then for any $ m \in \mathcal{I}(\tilde{\omega}) \setminus \mathcal{N}_{\Omega,\tilde{\omega}} $, the energy ground state with mass $ m $ will not be an action ground state for any $ \omega $.
	\end{theorem}
The above non-equivalence result may seem to be a bit `dry' for applications at this moment. However, we will use it to  discover the numerical evidence of the non-equivalence. In fact, we will find out later through numerical investigations that the parameter $\omega$ controls the occurrence of vortices in the ground state solution (see \cref{fig:Omgcp3}).
This is not revealed before in studies on energy ground states, where the quantized vortices are considered to be affected by the rotating velocity $\Omega$ but not chemical potential.
The phase transition in the action ground state when $\omega$ varies will lead to non-empty $ \mathcal{I}(\tilde{\omega}) \setminus \mathcal{N}_{\Omega,\tilde{\omega}}$ described in \cref{thm:Non-equivalence}.

\subsection{Alternative relation}
       One alternative description for the relation between the two ground states is to view the action ground state problem as a kind of dual problem of the energy ground state. Let us explain this in the following.        Consider the problem of the energy ground state:
		\begin{equation}\label{eq:Em_def}
			\mathcal{E}(m) := \min_{\substack{\phi \in X \\ \| \phi \|_{L^2}^2 = m}} E(\phi).
		\end{equation}
		Define the Lagrangian function $ L(\phi, \omega): X \times (-\infty, -\lambda_0(\Omega)) \rightarrow \mathbb{R} $ as
		\begin{equation*}
			L(\phi, \omega) = E(\phi) + \omega \left( \| \phi \|_{L^2}^2 - m \right).
		\end{equation*}
		Then the dual problem reads
			\begin{align}
				\mathcal{E}^\ast(m)
				&:= \sup_{\omega<-\lambda_0(\Omega)} \inf_{\phi \in X} L(\phi, \omega) \nonumber\\
				&= \sup_{\omega<-\lambda_0(\Omega)} \inf_{\phi \in X} \left[ S_{\Omega,\omega}(\phi) - \omega m \right] = \sup_{\omega<-\lambda_0(\Omega)} \left[ S^g_{\Omega}(\omega) - \omega m \right]. \label{dual problem}
			\end{align}
		Recalling \cref{eq:Em_def,eq:Sg_Omega_def}, we have
		\begin{align*}
			\mathcal{E}(m) &= \min_{\substack{\phi \in X \\ \| \phi \|_{L^2}^2 = m}} E(\phi) = \min_{\substack{\phi \in X \\ \| \phi \|_{L^2}^2 = m}}  \left[ S_{\Omega,\omega}(\phi) - \omega m \right]\\
   &\geq \min_{\substack{\phi \in X \\ \| \phi \|_{L^2}^2 = m}}  \left[ S^g_{\Omega}(\omega) - \omega m \right] = S^g_{\Omega}(\omega) - \omega m.
		\end{align*}
		It follows immediately that
		\begin{equation*}
			\mathcal{E}^\ast(m) \leq \mathcal{E}(m).
		\end{equation*}

        By \cref{cor:property_I}, for any $m>0$, there exists a unique $\omega \in (-\infty, -\lambda_0(\Omega))$ such that $m \in \mathcal{I}(\omega)$.

	\begin{theorem}\label{thm:dual}
            For $m \in \mathcal{I}(\omega)$, if $ m \in \mathcal{N}_{\Omega,\omega} $, then the dual problem is equivalent to the original problem in the sense that
		  \begin{equation*}
			\mathcal{E}^\ast(m) = \mathcal{E}(m).
		  \end{equation*}
	\end{theorem}	
	\begin{proof}
            Note that $ S^g_{\Omega}(\omega) - \omega m $ is a concave function of $ \omega $. Then, the supremum in \cref{dual problem} is obtained at $ \omega $ since $ m \in \mathcal{I}(\omega) = \partial S^g_{\Omega}(\omega) $, where $ \partial $ is the subdifferential operator. Since $ m \in \mathcal{N}_{\Omega,\omega} $, there exists $\phi_{\Omega, \omega} \in \mathcal{A}_{\Omega, \omega}$ satisfying $\| \phi_{\Omega, \omega} \|_{L^2}^2 = m$ which, by \cref{thm:actiontoenergy_new}, is also an energy ground state with mass $m$. Then
        \begin{equation}
            \mathcal{E}^\ast(m) = S^g_{\Omega}(\omega) - \omega m = E(\phi_{\Omega, \omega}) = \mathcal{E}(m),
        \end{equation}
        which completes the proof.
	\end{proof}

The rest of the section will be devoted to analyzing the action ground state problem in some limiting regimes of the parameters.

\subsection{Convergence of \texorpdfstring{$ \Omega \rightarrow 0 $}{Omega goes to 0}}
 In this subsection, we consider the case when the rotating speed $\Omega$ goes to zero.  Recall that $ \lambda_0(\Omega) \leq \omega_0 $ and $\lambda_0(\Omega) \rightarrow \omega_0 $ as $ \Omega \rightarrow 0$. We can provide the following rigorous convergence result.

	\begin{theorem}\label{thm:defocusing}
            Assume $ \omega < - \omega_0 $, then $ |\phi_{\Omega,\omega}| \rightarrow \phi_\omega $ in $ X $ as $ \Omega \rightarrow 0 $, where $ \phi_\omega $ is the unique nonnegative action ground state associated with $\omega$ when $ \Omega = 0 $.
	\end{theorem}
	
	The proof will be given with the help of the following lemmas. We recall the simplified notations \cref{eq:simplified_not} to be used here. 
        \begin{lemma}\label{lem:defocusingOmega0}
		When $ \omega < - \omega_0 $, we have
		\begin{enumerate}[label=(\roman*)]
			\item for $ \delta>0 $ sufficiently small, $ \{\phi_{\Omega,\omega}\}_{0 \leq \Omega < \delta} $ is bounded in $ X $;
			\item $ \lim_{\Omega \rightarrow 0} K_{\omega}(\phi_{\Omega,\omega}) = 0 $;
			\item $ \lim_{\Omega \rightarrow 0} \| \phi_{\Omega,\omega} \|_{L^{p+1}} = \| \phi_\omega \|_{L^{p+1}} $.
		\end{enumerate}
	\end{lemma}
	
	\begin{proof}
            By \cref{eq:S_negative}, noting that $ L_{\Omega}(\phi_\omega) = 0 $ since $\phi_\omega$ is real-valued, we have
            \begin{equation}\label{eq:S_relation_rotation}
                S_{\Omega, \omega}(\phi_{\Omega, \omega}) \leq S_{\Omega, \omega}(\phi_\omega) = S_\omega(\phi_\omega)<0,
            \end{equation}
            which proves (i) by \cref{lem:uniform_bound}.

            To show (ii), recalling that
            \begin{equation*}
                K_{\omega}(\phi_{\Omega,\omega}) = K_{\Omega, \omega}(\phi_{\Omega,\omega}) - L_\Omega(\phi_{\Omega,\omega}) = -L_\Omega(\phi_{\Omega,\omega}),
            \end{equation*}
            it suffices to show that
            \begin{equation}\label{eq:LOmega}
                \lim_{\Omega \rightarrow 0} L_\Omega(\phi_{\Omega,\omega}) = 0,
            \end{equation}
            which is a direct consequence of (i) and \cref{lem:control_rotation}.

            Then we prove (iii). By \cref{eq:S_relation_rotation}, recalling \cref{eq:K_def}, we have
            \begin{equation}\label{leq1de}
			\| \phi_{\Omega,\omega} \|_{L^{p+1}}^{p+1} \geq \| \phi_\omega \|_{L^{p+1}}^{p+1}.
		\end{equation}
            By \cref{eq:K_def,eq:omega_0_def}, we have
            \begin{equation*}
                0 = K_{\Omega, \omega} (\phi_{\Omega, \omega}) = Q_{\omega} (\phi_{\Omega, \omega}) + L_\Omega(\phi_{\Omega, \omega}) + \| \phi_{\Omega, \omega} \|_{L^{p+1}}^{p+1},
            \end{equation*}
            which implies by \cref{leq1de}
            \begin{equation*}
                Q_{\omega} (\phi_{\Omega, \omega}) = -L_\Omega(\phi_{\Omega, \omega}) - \| \phi_{\Omega, \omega} \|_{L^{p+1}}^{p+1} \leq -L_\Omega(\phi_{\Omega, \omega}) - \| \phi_{\omega} \|_{L^{p+1}}^{p+1}.
            \end{equation*}
            Recalling \cref{eq:LOmega}, we get $Q_{\omega} (\phi_{\Omega, \omega}) < 0$ when $0<\Omega<\delta'$ for some $\delta'$ small enough. It follows that, when $0<\Omega<\delta'$, there exists
		\begin{equation}\label{eq:lambda}
			\lambda = \left(\frac{-Q_\omega(\phi_{\Omega,\omega})}{\| \phi_{\Omega,\omega} \|_{L^{p+1}}^{p+1}} \right)^{\frac{1}{p-1}} = \left( 1 - \frac{K_{\omega}(\phi_{\Omega,\omega})}{\| \phi_{\Omega,\omega} \|_{L^{p+1}}^{p+1}} \right)^{\frac{1}{p-1}} 
		\end{equation}
		such that $K_\omega(\lambda \phi_{\Omega, \omega}) = 0$. Then we have                  $S_\omega(\phi_\omega) \leq S_\omega(\lambda \phi_{\Omega, \omega})$, which further implies, by recalling \cref{eq:K_def},
            \begin{equation}\label{leq2de}
			\| \phi_\omega \|_{L^{p+1}}^{p+1} \geq \| \lambda \phi_{\Omega,\omega} \|_{L^{p+1}}^{p+1}.
		\end{equation}
		Combine \cref{leq1de,leq2de} yields
		\begin{equation}\label{leqleqde}
			\| \phi_\omega \|_{L^{p+1}}^{p+1} \leq \| \phi_{\Omega,\omega} \|_{L^{p+1}}^{p+1} \leq \frac{\| \phi_\omega \|_{L^{p+1}}^{p+1}}{\lambda^{p+1}}.
		\end{equation}
		From \cref{eq:lambda}, recalling $K_{\omega}(\phi_{\Omega,\omega})\to 0$ as $\Omega          \rightarrow 0$ established in (ii), we have $ \lambda \rightarrow 1 $ as $ \Omega \rightarrow 0 $, which           proves (iii) by \cref{leqleqde}.
	\end{proof}

\begin{lemma}\label{lem:aux defocusing}
    For any $ \phi \in X $, if $ \| \phi \|_{L^{p+1}}^{p+1} = \| \phi_{\Omega, \omega} \|_{L^{p+1}}^{p+1} $ and $ K_{\Omega,\omega}(\phi) \leq 0 $, then $ \phi \in \mathcal{A}_{\Omega, \omega} $.
\end{lemma}
\begin{proof}
    Since $K_{\Omega,\omega}(\phi)\leq0$, recalling \cref{eq:K_def}, there exists $$ \lambda = \left( -Q_{\Omega, \omega}(\phi)/\| \phi \|_{L^{p+1}}^{p+1} \right)^{1/(p-1)} \geq 1 $$ such that $ K_{\Omega,\omega}(\lambda \phi) = 0 $. Since $\phi_{\Omega, \omega}$ minimizes $S_{\Omega, \omega}$ on $X$, we have, by $K_{\Omega,\omega}(\lambda \phi) = 0$, $\lambda \geq 1$ and $\| \phi \|_{L^{p+1}}^{p+1} = \| \phi_{\Omega, \omega} \|_{L^{p+1}}^{p+1}$,
    \begin{equation*}
        S_{\Omega,\omega}(\phi_{\Omega, \omega}) \leq S_{\Omega,\omega}(\lambda \phi) = -\frac{p-1}{p+1} \| \lambda \phi \|_{L^{p+1}}^{p+1} \leq -\frac{p-1}{p+1} \| \phi_{\Omega, \omega} \|_{L^{p+1}}^{p+1} = S_{\Omega,\omega}(\phi_{\Omega, \omega}),
    \end{equation*}
    which implies $ \lambda = 1 $ and $S_{\Omega, \omega}(\phi) = S_{\Omega,\omega}(\phi_{\Omega, \omega})$. Hence, $\phi$ is an action ground state associated with $\omega$ and $\Omega$.
\end{proof}

	\begin{proof}[Proof of \cref{thm:defocusing}]
            With the result (i) in \cref{lem:defocusingOmega0}, for $\delta>0$ sufficiently small, one has the set $ \{\phi_{\Omega,\omega}\}_{0 \leq \Omega < \delta} $ is bounded in $ X $ and so is $  \{|\phi_{\Omega,\omega}|\}_{0 \leq \Omega < \delta} $ (by noting that the inequality $\|\nabla|\phi_{\Omega,\omega}|\|_{L^2}\leq \|\nabla\phi_{\Omega,\omega}\|_{L^2}$).
            Then for any sequence $\Omega^k \rightarrow 0 $, there exists a subsequence (still denoted by $\{\Omega^k\}$) and some $ \phi \in X $ such that $ |\phi_{\Omega^k,\omega}| \rightharpoonup \phi $ in $ X $. By the weak lower-semicontinuity of the norm $\| \cdot \|_X$, \cref{lem:Xembed} and (iii) in \cref{lem:defocusingOmega0}, we have
		\begin{align}
			&\| \phi \|_X \leq \liminf_{k \rightarrow \infty} \| |\phi_{\Omega^k,\omega}| \|_X \leq \liminf_{k \rightarrow \infty} \| \phi_{\Omega^k,\omega} \|_X, \label{2.17}\\
			&\| \phi \|_{L^{p+1}}^{p+1} = \lim_{k \rightarrow \infty} \| |\phi_{\Omega^k,\omega}| \|_{L^{p+1}}^{p+1} = \| \phi_\omega \|_{L^{p+1}}^{p+1}, \quad \| \phi \|_{L^2}^{2} = \lim_{k \rightarrow \infty} \| |\phi_{\Omega^k,\omega}| \|_{L^2}^{2}. \label{eq:4.33}
		\end{align}
		It follows that
		\begin{equation*}
			K_{\omega}(\phi) \leq \liminf_{k \rightarrow \infty} K_{\omega}(\phi_{\Omega^k,\omega}) = 0.
		\end{equation*}
            Since $ \| \phi \|_{L^{p+1}}^{p+1} = \| \phi_\omega \|_{L^{p+1}}^{p+1} $,
            we conclude by \cref{lem:aux defocusing} with $\Omega = 0$ that $K_{\omega}(\phi)=0$ and $ \phi $ is also an action ground state in the non-rotating case associated with $\omega$.
            By \cref{lem:uniqueness} $, |\phi| = \phi_\omega $, which implies, by \cref{eq:4.33} and the arbitrariness of the sequence $\{\Omega^k\}$,
            \begin{equation}\label{L2limit}
			\lim_{\Omega \rightarrow 0} \| \phi_{\Omega,\omega} \|_{L^2} = \| \phi_\omega \|_{L^2}.
		  \end{equation}
            By \cref{L2limit}, (ii) and (iii) in \cref{lem:defocusingOmega0}, we have $ \lim_{\Omega \rightarrow 0} \|\phi_{\Omega,\omega}\|_X = \| \phi_\omega \|_X $. Noting \cref{2.17}, we also have $ \lim_{\Omega \rightarrow 0} \| | \phi_{\Omega,\omega} | \|_X = \| \phi_\omega \|_X $, which together with the weak convergence, yields $ |\phi_{\Omega,\omega}| \rightarrow \phi_{\omega} $ as $\Omega \rightarrow 0$ in $ X $.
	\end{proof}

{In fact, from our numerical results in \cref{sec:Vortices}, we can observe that when $\Omega$ is smaller than a critical value $\Omega^\text{c}$, no vortices appear in the action ground state and $\phi_{\Omega,\omega}$ can be chosen as a nonnegative function that coincides with $\phi_\omega$. }

\subsection{Asymptotics as \texorpdfstring{$ \omega \rightarrow -\lambda_0(\Omega)^- $}{omega-to-minus-lambda0}}
 Next, we consider the limits for  $\omega\in(-\infty,-\lambda_0(\Omega))$ and we begin with the upper limit, i.e., $ \omega \rightarrow -\lambda_0(\Omega)^- $.
    For the linear Schr\"odinger operator with rotation $ R = - \frac12\Delta + V - \Omega L_z $, under Assumption \ref{lem:S-welldef}, it is known to be a self-adjoint operator on $ L^2 $ with purely discrete spectrum\cite{Hajaiej,spectrum1996,spectrum1994}.  Let $ E_0 $ be the eigenspace of $ R $ associated with $ \lambda_0(\Omega) $ (the smallest eigenvalue of $ R $).
	
	\begin{theorem}\label{thm:smallomega_defocusing}
		Let $ \hat \phi_{\Omega,\omega} = \phi_{\Omega,\omega}/\| \phi_{\Omega,\omega} \|_{L^2} $. Then we have
		\begin{equation*}
			\lim_{\omega \rightarrow -\lambda_0(\Omega)^-} d(\hat \phi_{\Omega,\omega}, E_0) = 0,
		\end{equation*}
		where
		\begin{equation*}
			d(\hat \phi_{\Omega,\omega}, E_0) = \inf_{\phi \in E_0} \left\| \hat \phi_{\Omega,\omega}  - \phi  \right\|_X.
		\end{equation*}
	\end{theorem}
	
	\begin{proof}
		Let $ X $ be equipped with an inner product $ \lrang{\cdot}{\cdot}_{R(\kappa)} $. Note that $ K_{\Omega,\omega}(\phi_{\Omega,\omega}) = 0 $ implies that
		\begin{equation}\label{eq:I=0 defocusing}
			\frac{\lrang{R \phi_{\Omega,\omega}}{\phi_{\Omega,\omega}}}{\| \phi_{\Omega,\omega} \|_{L^2}^2} + \omega + \frac{\| \phi_{\Omega,\omega} \|_{L^{p+1}}^{p+1}}{\| \phi_{\Omega,\omega} \|_{L^2}^2} = 0.
		\end{equation}
		By \cref{eq:S_negative}, one has
		\begin{equation*}
			\begin{aligned}
				0>S_{\Omega, \omega}(\phi_{\Omega,\omega})
				&= \lrang{R \phi_{\Omega,\omega}}{\phi_{\Omega,\omega}} + \omega \| \phi_{\Omega,\omega} \|_{L^2}^2 + \frac{2}{p+1}\| \phi_{\Omega,\omega} \|_{L^{p+1}}^{p+1}\\
    &\geq (\omega + \lambda_0(\Omega)) \| \phi_{\Omega,\omega} \|_{L^2}^2.
			\end{aligned}
		\end{equation*}
		Recalling \cref{eq:I=0 defocusing}, we have
		\begin{equation*}
			0>-\frac{p-1}{p+1} \frac{\| \phi_{\Omega,\omega} \|_{L^{p+1}}^{p+1}}{\| \phi_{\Omega,\omega} \|_{L^2}^2} \geq \omega + \lambda_0(\Omega),
		\end{equation*}
        which implies
        \begin{equation}
            \frac{\| \phi_{\Omega,\omega} \|_{L^{p+1}}^{p+1}}{\| \phi_{\Omega,\omega} \|_{L^2}^2} \rightarrow 0 \text{ as } \omega \rightarrow -\lambda_0(\Omega)^-.
        \end{equation}
		By \cref{eq:I=0 defocusing}, $ \phi_{\Omega,\omega} $ is a minimizing sequence for \cref{inf:linear with rotation} as $ \omega \rightarrow -\lambda_0(\Omega)^{-} $. Noting that
		$$ \| \hat \phi_{\Omega,\omega} \|_{R(\kappa)}^2 = \frac{\lrang{R \phi_{\Omega,\omega}}{\phi_{\Omega,\omega}}}{\| \phi_{\Omega,\omega} \|_{L^2}^2} + \kappa \rightarrow \lambda_0(\Omega) + \kappa \text{ as }\ \omega \rightarrow -\lambda_0(\Omega)^-, $$
        we have $ \{ \hat \phi_{\Omega,\omega} \}_{\omega} $ is bounded in $ X $ when $ \omega $ close to $ -\lambda_0(\Omega) $. For any sequence $ \omega^k \rightarrow -\lambda_0(\Omega)^- $, there exists a subsequence $ \omega^{n_k} $ such that for some $ \phi \in X $,
		\begin{equation*}
			 \hat \phi_{\Omega,\omega^{n_k}} \rightharpoonup \phi \text{ in } X, \quad \hat \phi_{\Omega,\omega^{n_k}} \rightarrow \phi \text{ in } L^2, \quad \text{ as } k \rightarrow \infty.
		\end{equation*}
		Then $ \| \phi \|_{L^2} = 1 $ and
		\begin{equation*}
			\lrang{R \phi}{\phi} + \kappa = \| \phi \|_{R(\kappa)}^2 \leq \liminf_{k \rightarrow \infty} \| \hat \phi_{\Omega,\omega^{n_k}} \|_{R(\kappa)}^2 = \kappa + \lambda_0(\Omega),
		\end{equation*}
		which implies that $ \lrang{R \phi}{\phi} = \lambda_0(\Omega) $. It follows immediately that $\| \hat \phi_{\Omega,\omega^{n_k}} \|_{R(\kappa)} \rightarrow \| \phi \|_{R(\kappa)}$ as $\omega \rightarrow -\lambda_0(\Omega)$, which, together with the weak convergence, implies $ \hat \phi_{\Omega,\omega^{n_k}} \rightarrow \phi $ in $ X $. Since $\phi\in E_0$, one obtain $d(\hat \phi_{\Omega,\omega^{n_k}}, E_0)\leq\|\hat \phi_{\Omega,\omega^{n_k}} - \phi \|_X \to 0$ as $k\to\infty$. By the arbitrariness of the sequence $\{\omega^k\}$, the conclusion follows immediately.
	\end{proof}

        Following a similar idea in Ref.~\refcite{Fukuizumi2}, we can precisely show the asymptotics for mass.
	\begin{theorem}(Asymptotic for mass)\label{thm:smallomega_asymptotics}
		When $ \omega \rightarrow -\lambda_0(\Omega)^- $,  the mass of the action ground state is approaching zero at  the rate
		\begin{equation*}
			\| \phi_{\Omega,\omega} \|_{L^2}^2 \sim |\omega + \lambda_0(\Omega)|^\frac{2}{p-1}.
		\end{equation*}
	\end{theorem}

	\begin{proof}
		Without loss of generality, we assume that $ \lambda_0(\Omega) = 0 $. Then $ E_0 $ is the eigenspace of $ R $ associated with eigenvalue $ 0 $, which is finite dimensional and thus all the norms on it are equivalent. We shall show that $ \| \phi_{\Omega,\omega} \|_{L^2}^2 \sim |\omega|^\frac{2}{p-1} $.
		
		Assume that $ \phi_{\Omega,\omega} = \phi^0_\omega + y_\omega $, where $ \phi^0_\omega \in E_0 $ and $ \lrang{\phi^0_\omega}{y_\omega}_{R(1)} = 0 $, i.e., $ \phi^0_\omega $ is the orthogonal projection of $ \phi_{\Omega,\omega} $ from $ X $ onto $ E_0 $ with respect to $ \lrang{\cdot}{\cdot}_{R(1)} $. By the min-max principle,
	\begin{equation}\label{eq:min_max}
            \lambda_1 \| y_\omega \|_{L^2}^2 \leq \lrang{R y_\omega}{y_\omega},\  \text{ for some } \lambda_1 > 0.
	\end{equation}
		Since $ \phi^0_\omega \in E_0 $, one has $ R \phi^0_\omega = 0 $ and thus
		\begin{equation*}
			0 = \lrang{\phi^0_\omega}{y_\omega}_{R(1)} = \lrang{R \phi^0_\omega}{y_\omega} + \lrang{\phi^0_\omega}{y_\omega} = \lrang{\phi^0_\omega}{y_\omega},
		\end{equation*}
		which implies $ \phi^0_\omega $ and $ y_\omega $ are also orthogonal in $ L^2 $. It follows that
		\begin{equation*}
			\| \phi_{\Omega,\omega} \|_{R(1)}^2 = \| \phi^0_\omega \|_{R(1)}^2 + \| y_\omega \|_{R(1)}^2, \quad \| \phi_{\Omega,\omega} \|_{L^2}^2 = \| \phi^0_\omega \|_{L^2}^2 + \| y_\omega \|_{L^2}^2.
		\end{equation*}
            By \cref{limit_of_mass}, we have
            \begin{equation*}
                \| y_\omega \|_{L^2} \rightarrow 0, \quad \| \phi_\omega^0 \|_{L^2} \rightarrow 0, \quad \text{ as } \quad \omega \rightarrow 0^-.
		\end{equation*}
		We claim that
            \begin{equation*}
                \| y_\omega \|_{L^2} = o(\| \phi_\omega^0 \|_{L^2}), \quad \text{ as } \quad \omega \rightarrow 0^-.
		\end{equation*}
		In fact, \cref{thm:smallomega_defocusing} implies that
		\begin{equation}\label{eq:distance_E_0}
                \frac{\left\| \phi_{\Omega,\omega} - \phi_\omega^0 \right \|_{R(1)}}{\| \phi_{\Omega,\omega} \|_{L^2}}= \inf_{\phi \in E_0} \left\| \hat \phi_{\Omega,\omega}  - \phi  \right\|_{R(1)} \lesssim \inf_{\phi \in E_0} \left\| \hat \phi_{\Omega,\omega}  - \phi  \right\|_X  \rightarrow 0.
		\end{equation}
		Since
		\begin{equation*}
			\frac{\left\| \phi_{\Omega,\omega} - \phi_\omega^0 \right \|_{R(1)}^2}{\| \phi_{\Omega,\omega} \|_{L^2}^2} = \frac{\| y_\omega \|_{R(1)}^2}{\| \phi^0_\omega \|_{L^2}^2 + \| y_\omega \|_{L^2}^2} \geq \frac{C\| y_\omega \|_{L^2}^2}{\| \phi^0_\omega \|_{L^2}^2 + \| y_\omega \|_{L^2}^2},
		\end{equation*}
		one has $ \| y_\omega \|_{L^2}^2 / \| \phi^0_\omega \|_{L^2}^2 \rightarrow 0 $ as $ \omega \rightarrow 0^- $. Then the claim is proved.
		
		Noting $ \phi_{\Omega,\omega} $ satisfies the stationary RNLS \cref{model0}, we have
		\begin{equation}\label{SNLS}
			R \phi_{\Omega,\omega} - |\omega| \phi_{\Omega,\omega} = -|\phi_{\Omega,\omega}|^{p-1} \phi_{\Omega,\omega} =: -g(\phi_{\Omega,\omega}).
		\end{equation}
		Multiplying both sides of \cref{SNLS} with $ \overline{y_\omega} $ and integrating over $ \mathbb{R}^d $, we obtain
		\begin{align}
				&\| y_\omega \|_{R(1)}^2 -(1+|\omega|) \| y_\omega \|_{L^2}^2
= |\lrang{g(\phi_{\Omega,\omega})}{y_\omega}|\nonumber \\
\leq& |\lrang{g(\phi_{\Omega,\omega}) - g(\phi^0_\omega)}{y_\omega} | + | \lrang{g(\phi^0_\omega)}{y_\omega} | \nonumber\\
\leq& C(\| \phi_{\Omega,\omega} \|_{L^{p+1}}^{p-1} + \| \phi^0_\omega \|_{L^{p+1}}^{p-1} ) \| y_\omega \|_{L^{p+1}}^2 + \| \phi^0_\omega \|_{L^{p+1}}^{p} \| y_\omega \|_{L^{p+1}} \nonumber\\
\lesssim& (\| \phi_{\Omega,\omega} \|_{L^{p+1}}^{p-1} + \| \phi^0_\omega \|_{L^2}^{p-1} ) \| y_\omega \|_{R(1)}^2 + \| \phi^0_\omega \|_{L^2}^{p} \| y_\omega \|_{R(1)}.\label{eq:y_omega}
		\end{align}
        Besides, when $ |\omega| < \varepsilon $ with $\varepsilon$ sufficiently small, one has, by \cref{eq:min_max},
		\begin{equation*}
			\| y_\omega \|_{R(1)}^2 - (1+|\omega|) \| y_\omega \|_{L^2}^2 \geq \| y_\omega \|_{R(1)}^2 - \frac{1+|\omega|}{1 + \lambda_1} \| y_\omega \|_{R(1)}^2 \geq \frac{\lambda_1 - \varepsilon}{1+\lambda_1} \| y_\omega \|_{R(1)}^2.
		\end{equation*}
		Noting that $ \| \phi_{\Omega,\omega} \|_{L^{p+1}} + \| \phi^0_\omega \|_{L^2} \rightarrow 0 $ as $ \omega \rightarrow 0^- $, then it follows from \cref{eq:y_omega} that
		\begin{equation}\label{est:y_omgea}
			\| y_\omega \|_{R(1)} \lesssim \| \phi^0_\omega \|_{L^2}^p.
		\end{equation}
		Hence, $ \| y_\omega \|_{R(1)} = o(\| \phi^0_\omega \|_{L^2}) $ and $ \| \phi_{\Omega,\omega} \|_{R(1)} \sim \| \phi^0_\omega \|_{R(1)} \sim \| \phi^0_\omega \|_{L^2} $. Similarly, by multiplying both sides of \cref{SNLS} with $ \overline{\phi^0_\omega} $ and integrating over $ \mathbb{R}^d $, one can obtain that
		\begin{equation}\label{eq:phi0}
			\begin{aligned}
				|\omega| \| \phi^0_\omega \|_{L^2}^2
				&= \lrang{g(\phi_{\Omega,\omega})}{\phi^0_\omega} = \lrang{g(\phi_{\Omega,\omega}) - g(\phi^0_\omega)}{\phi^0_\omega} + \lrang{g(\phi^0_\omega)}{\phi^0_\omega} \\
				&= \| \phi^0_\omega \|_{L^{p+1}}^{p+1} + \lrang{g(\phi_{\Omega,\omega}) - g(\phi^0_\omega)}{\phi^0_\omega}.
			\end{aligned}
		\end{equation}
		Note that
		\begin{equation*}
			\begin{aligned}
				|\lrang{g(\phi_{\Omega,\omega}) - g(\phi^0_\omega)}{\phi^0_\omega}|
				&\leq C(\| \phi_{\Omega,\omega} \|_{L^{p+1}}^{p-1} + \| \phi^0_\omega \|_{L^{p+1}}^{p-1} ) \| y_\omega \|_{L^{p+1}} \| \phi^0_\omega \|_{L^{p+1}} \\
				&\lesssim (\| \phi_{\Omega,\omega} \|_{R(1)}^{p-1} + \| \phi^0_\omega \|_{L^2}^{p-1} ) \| y_\omega \|_{R(1)} \| \phi^0_\omega \|_{L^{2}}
				\lesssim \| \phi^0_\omega \|_{L^2}^{2p}.
			\end{aligned}
		\end{equation*}
		Then it follows from \cref{eq:phi0} that
		$
			|\omega| \| \phi^0_\omega \|_{L^2}^2 \sim \| \phi^0_\omega \|_{L^{p+1}}^{p+1} \sim \| \phi^0_\omega \|_{L^2}^{p+1},
		$
		which implies $ \| \phi^0_\omega \|_{L^2}^2 \sim |\omega|^\frac{2}{p-1} $ and completes the proof.
	\end{proof}
	
	\begin{corollary}\label{cor:smallomega_asymp_distE0}
		One has	$d(\hat \phi_{\Omega,\omega}, E_0) \lesssim |\omega + \lambda_0(\Omega)|$ as $\omega \rightarrow -\lambda_0(\Omega)^-$.
	\end{corollary}
	
	\begin{proof}
		Recalling \cref{eq:distance_E_0}, we have
		\begin{equation*}
			d(\hat \phi_{\Omega,\omega}, E_0) = \inf_{\phi \in E_0} \| \hat \phi_{\Omega,\omega} - \phi \|_X \sim \inf_{\phi \in E_0} \| \hat \phi_{\Omega,\omega} - \phi \|_{R(1)} = \frac{ \| y_\omega \|_{R(1)}}{\| \phi_{\Omega,\omega} \|_{L^2}},
		\end{equation*}
		where $y_\omega$ is the same as that defined in the proof of \cref{thm:smallomega_asymptotics}.
            By \cref{est:y_omgea,thm:smallomega_asymptotics}, we have
		\begin{equation*}
                \frac{ \| y_\omega \|_{R(1)}}{\| \phi_{\Omega,\omega} \|_{L^2}} \lesssim \frac{ \| \phi^0_\omega \|_{L^2}^p }{\| \phi_{\Omega,\omega} \|_{L^2}} \sim \| \phi^0_\omega \|_{L^2}^{p-1} \sim |\omega + \lambda_0(\Omega)|.
		\end{equation*}
            The proof is thus completed.
	\end{proof}
	
	\begin{corollary}(Asymptotic for action)\label{cor:smallomega_asymp_Sg}
		$ S^g_\Omega(\omega) \sim -| \omega + \lambda_0(\Omega) |^\frac{p+1}{p-1} $ as $ \omega \rightarrow - \lambda_0(\Omega)^- $.
	\end{corollary}
	\begin{proof}
	    Recalling that $S^g_\Omega(\omega) = -\frac{p-1}{p+1}\| \phi_{\Omega, \omega}  \|_{L^{p+1}}^{p+1}$, the result is a direct consequence of \cref{thm:smallomega_defocusing,thm:smallomega_asymptotics}.
	\end{proof}

\subsection{Asymptotics as \texorpdfstring{$ \omega \rightarrow -\infty $}{omega infinity}}\label{subsec:omega infinity}
In this section, we consider the limit $\omega \rightarrow -\infty$. Here, we only present some formal computations and the results will be further explored by the numerical experiments in \cref{sec:numer}. 

We define a rescaled function $ \tilde \phi_{\Omega,\omega} $ as
	\begin{equation}\label{rescaled state}
		\phi_{\Omega,\omega}(\sqrt{|\omega|}\,\mathbf{x}) = |\omega|^{\frac{1}{p-1}} \tilde \phi_{\Omega,\omega} (\mathbf{x}).
	\end{equation}
	Then one has
	\begin{align*}
		& \| \nabla \phi_{\Omega,\omega} \|_{L^2}^2 = |\omega|^{\frac{2}{p-1}-1+\frac{d}{2}} \| \nabla \tilde \phi_{\Omega,\omega} \|_{L^2}^2, \quad
		 \omega \| \phi_{\Omega,\omega} \|_{L^2}^2 = -|\omega|^{\frac{2}{p-1}+1+\frac{d}{2}} \| \tilde \phi_{\Omega,\omega} \|_{L^2}^2, \\
		& \| \phi_{\Omega,\omega} \|_{L^2_V}^2 = |\omega|^{\frac{2}{p-1}+\frac{d}{2}} \int_{\mathbb{R}^d} V(|\omega|^\frac{1}{2} \mathbf{x}) |\tilde \phi_{\Omega,\omega}(\mathbf{x})| \mathrm{d} \mathbf{x}, \\
		&\| \phi_{\Omega,\omega} \|_{L^{p+1}}^{p+1} = |\omega|^{\frac{2}{p-1}+1+\frac{d}{2}} \| \tilde \phi_{\Omega,\omega} \|_{L^{p+1}}^{p+1}, \quad L_\Omega(\phi_{\Omega,\omega}) = |\omega|^{\frac{2}{p-1}+\frac{d}{2}} L_\Omega(\tilde \phi_{\Omega,\omega}).
	\end{align*}
	If $ V $ is a harmonic oscillator, which is homogeneous of degree 2, then
	\begin{equation*}
		\| \phi_{\Omega,\omega} \|_{L^2_V}^2 = |\omega|^{2/(p-1)+1+d/2} \| \tilde \phi_{\Omega,\omega} \|_{L^2_V}^2.
	\end{equation*}
	Let $ \varepsilon = |\omega|^{-1} \ll 1 $. Then we have
	\begin{align*}
		&S_{\Omega,\omega}(\phi_{\Omega,\omega})\\
		&= \frac{1}{2} \| \nabla \phi_{\Omega,\omega} \|_{L^2}^2 + \| \phi_{\Omega,\omega} \|_{L^2_V}^2 \omega + \| \phi_{\Omega,\omega} \|_{L^2}^2 + \frac{2}{p+1} \| \phi_{\Omega,\omega} \|_{L^{p+1}}^{p+1} + L_\Omega(\phi_{\Omega,\omega}), \\
		&= |\omega|^{\frac{p+1}{p-1}+\frac{d}{2}} \bigg[ \frac{\varepsilon^2}{2} \| \nabla \tilde \phi_{\Omega,\omega} \|_{L^2}^2 + \| \tilde \phi_{\Omega,\omega} \|_{L^2_V}^2 - \| \tilde \phi_{\Omega,\omega} \|_{L^2}^2 + \frac{2}{p+1} \| \tilde \phi_{\Omega,\omega} \|_{L^{p+1}}^{p+1} \\
&\qquad\qquad\quad\ \ + \varepsilon L_\Omega(\tilde \phi_{\Omega,\omega}) \bigg].
	\end{align*}
	Note that $ \phi_{\Omega,\omega} $ is a minimizer of $ S_{\Omega,\omega}(\phi) $. It follows that $ \tilde \phi_{\Omega,\omega} $ is a minimizer of $ \tilde{S}_{\Omega,\omega} $, where
	\begin{equation*}
		\tilde{S}_{\Omega,\omega}(\phi) = \int_{\mathbb{R}^d} \left( \frac{\varepsilon^2}{2} |\nabla\phi|^2 - |\phi|^2 + V|\phi|^2 + \frac{2}{p+1} |\phi|^{p+1} - \varepsilon \Omega \overline{\phi}L_z \phi \right) \mathrm{d} \mathbf{x}.
	\end{equation*}
	Hence, $ \tilde \phi_{\Omega,\omega} $ satisfies
	\begin{equation*}
		- \frac{\varepsilon^2}{2} \Delta \phi + V \phi - \phi + |\phi|^{p-1}\phi - \varepsilon \Omega L_z \phi = 0.
	\end{equation*}
 Formally, as $ \omega \rightarrow -\infty $, we have
	\begin{equation}\label{TF limit}
		|\tilde \phi_{\Omega,\omega}| \rightarrow \tilde \phi,
	\end{equation}
	where $\tilde \phi$ is a Thomas-Fermi type approximation:
	\begin{equation*}
		\tilde \phi(\mathbf{x}) = \left\{\begin{aligned}
			&[(1 - V(\mathbf{x})]^\frac{1}{p-1}, && V(\mathbf{x}) \leq 1, \\
			&0, && V(\mathbf{x}) > 1.
		\end{aligned}
		\right.
	\end{equation*}

\section{Numerical explorations}\label{sec:numer}
The theoretical studies in the previous two sections show some qualitative and quantitative features of the action ground states but are not exhaustive. In this section, we shall verify and complement the theoretical investigations by numerical explorations. To accurately solve \cref{phi_g-def}, we consider a numerical scheme proposed and analyzed in our recent work\cite{LYZ}. It is a gradient flow method based on the unconstrained formulation \cref{eq:S_negative} discretized by a stabilized semi-implicit Fourier pseudospectral scheme, and it can effectively work for any given $\Omega,\omega$ within the mathematically valid range. We shall adopt this scheme for numerical computations throughout the section, and we refer the readers to Ref.~\refcite{LYZ} for its detailed presentation and implementation.

For simplicity, here we present only two dimensional examples, i.e., $d=2$, $\mathbf{x}=(x_1,x_2)^\top$ in (\ref{model0}),  and
the other parameters are fixed as $p=3,\beta=1$. The potential function is chosen as the harmonic oscillator $V(\mathbf{x})=\frac12|\mathbf{x}|^2$, and so $\Omega$ will be considered within the interval $[0,1)$. The numerical results for more kinds of potentials will be given in Ref.~\refcite{LYWZ}. Under the setup,  we have as a matter of fact (see Lemma~\ref{conj1-lambdaOmg}): $\forall|\Omega|<1$,
\[ \lambda_0(\Omega)=\inf\Big\{\int_{\mathbb{R}^2}\left(\frac12|\nabla \varphi|^2+\frac{1}{2}|\mathbf{x}|^2|\varphi|^2-\Omega \overline{\varphi}L_z\varphi\right)\rmd\mathbf{x}:\, \|\varphi\|_{L^2}=1\Big\}\equiv1=\omega_0. \]
Thus, the parameter $\omega$ can be considered in the interval $(-\infty,-1)$. For each $\omega$ in the non-rotating case, we denote $S_g(\omega)$ and $M_g(\omega)$ for the action value and the mass value at the ground state, respectively. In the rotating case, our algorithm searches for one possible ground state solution $\phi_{\Omega,\omega}$ for the given $\Omega$ and $\omega$, and we denote $S_g(\Omega,\omega)=S_{\Omega,\omega}(\phi_{\Omega,\omega})$ and
$M_g(\Omega,\omega)=\|\phi_{\Omega,\omega}\|_{L^2}^2$ for its action value and mass value, respectively.

\subsection{Behaviour of action and mass in \texorpdfstring{$\omega$}{omega}.}
We first investigate the behaviours of the action and the
mass of the ground state with respect to $\omega$. To do so, we take a fine mesh for $\omega\in(-20,-1)$ and compute the action ground state for each $\omega$. Then the action and the mass of the ground state are plotted as functions of  $\omega$.

\begin{figure}[h!]
\centering
\footnotesize
\includegraphics[width=.39\textwidth,height=.35\textwidth]{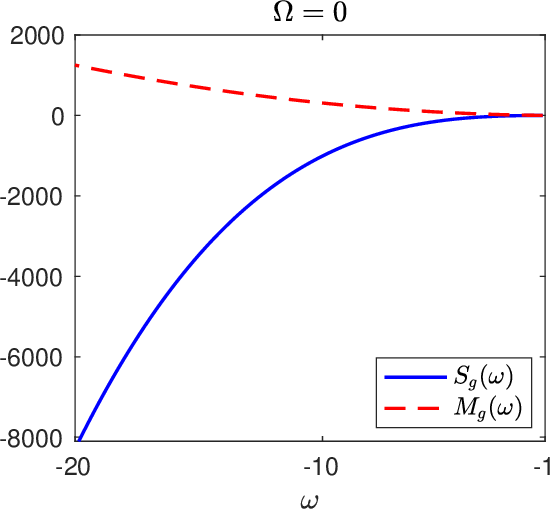} \qquad
\includegraphics[width=.39\textwidth,height=.35\textwidth]{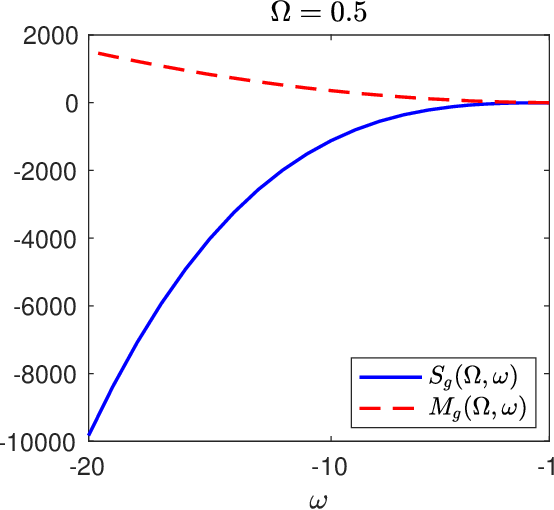}
\vspace{-1ex}
\caption{Change of the action $S_g(\Omega,\omega)$ and the mass $M_g(\Omega,\omega)$  with respect to $\omega$.
Left: $\Omega=0$; Right: $\Omega=0.5$.}
\label{fig:SgMgvs_w_Omg0_5}
\end{figure}

In the non-rotating case ($\Omega=0$), the numerical result shown in \cref{fig:SgMgvs_w_Omg0_5} (left) illustrates the change of the action and the mass with respect to $\omega$. From the figure, one observes that the ground state mass $M_g(\omega)$ is strictly monotonically decreasing with respect to $\omega\in(-\infty,-\omega_0)$ and gives a bijection from $(-\infty,-\omega_0)$ to $(0,\infty)$. This verifies \cref{lem1,lem2,lem3} and consequently,  \cref{thm:energytoaction}. In addition, the ground state action $S_g(\omega)$ is strictly monotonically increasing with respect to $\omega\in(-\infty,-\omega_0)$.

In the rotating case, where we take $\Omega=0.5$, the behaviours  are very similar. It is observed from \cref{fig:SgMgvs_w_Omg0_5} (right) that, for a fixed $\Omega$, $M_g(\Omega,\omega)$ is still strictly monotonically decreasing with respect to $\omega$ and has the following limits:
\[ M_g(\Omega,\omega)\to0 \quad \mbox{as }\omega\to-\lambda_0(\Omega)^-, \quad M_g(\Omega,\omega)\to \infty \quad \mbox{as }\omega\to-\infty. \]
For a fixed $\Omega$, $S_g(\Omega,\omega)$ is strictly monotonically increasing with respect to $\omega$ and
\[ S_g(\Omega,\omega)\to0 \quad \mbox{as }\omega\to-\lambda_0(\Omega)^-, \quad S_g(\Omega,\omega)\to -\infty \quad \mbox{as }\omega\to-\infty. \]

\begin{figure}[!h]
\centering
\footnotesize
\includegraphics[width=.39\textwidth,height=.35\textwidth]{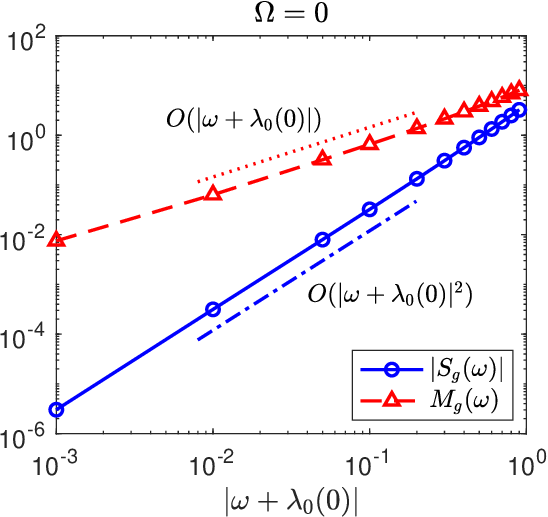} \qquad
\includegraphics[width=.39\textwidth,height=.35\textwidth]{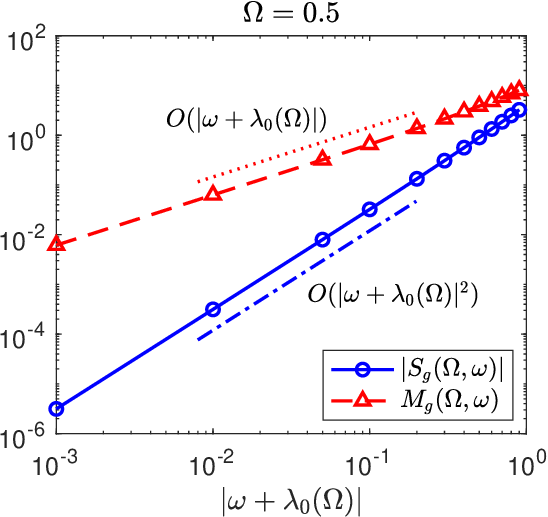}
\vspace{-2ex}
\caption{Asymptotic rates of  $S_g(\Omega,\omega)$ and $M_g(\Omega,\omega)$  when $\omega\to-\lambda_0(\Omega)^-$. Left: $\Omega=0$; Right: $\Omega=0.5$.}
\label{fig:SgMgAsymp_smallw_Omg0_5}
\end{figure}

\begin{figure}[!h]
\centering
\footnotesize
\includegraphics[width=.38\textwidth,height=.35\textwidth]{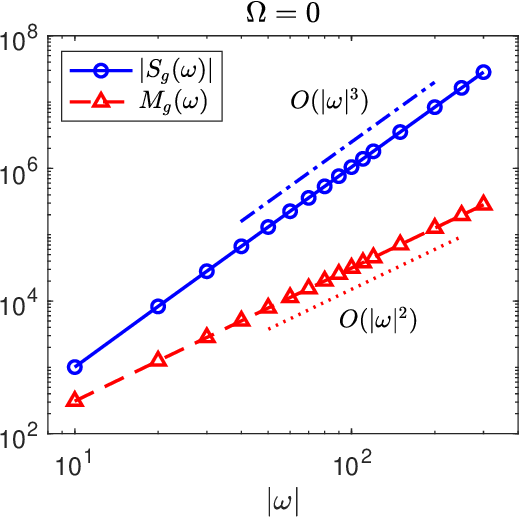} \qquad
\includegraphics[width=.38\textwidth,height=.35\textwidth]{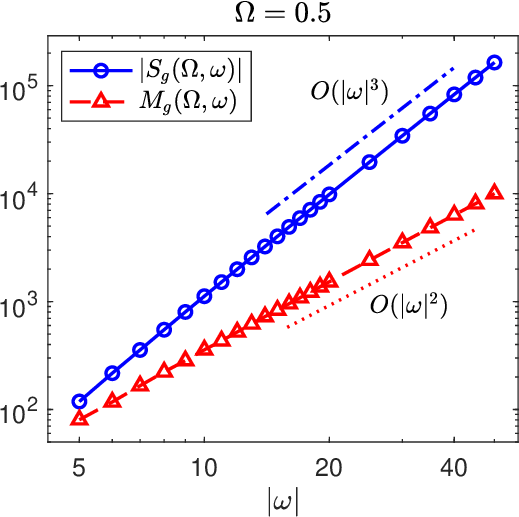}
\vspace{-2ex}
\caption{Asymptotic rates of $S_g(\Omega,\omega)$ and $M_g(\Omega,\omega)$ when $\omega\to -\infty$. Left: $\Omega=0$; Right: $\Omega=0.5$.}
\label{fig:SgMgAsymp_largew_Omg0and0_5}
\end{figure}

\subsection{Asymptotics for limiting \texorpdfstring{$\omega$}{omega}.}
Next, we investigate the asymptotic convergence rates of the ground state action and mass for limiting $\omega$.

Firstly, we consider the limit $\omega\to-\lambda_0(\Omega)^-$.  As shown in \cref{fig:SgMgAsymp_smallw_Omg0_5}, the following asymptotic rates for the action and mass at the ground state are observed in both non-rotating and rotating regimes:
\begin{align*}
 |S_g(\Omega,\omega)| \sim |\omega+\lambda_0(\Omega)|^2, \quad M_g(\Omega,\omega) \sim |\omega+\lambda_0(\Omega)|,\quad \mbox{as }\omega\to-\lambda_0(\Omega)^-.
\end{align*}
Clearly, this verifies \cref{thm:smallomega_asymptotics} and \cref{cor:smallomega_asymp_Sg} in the case of $p=3$.

Secondly, we consider the limit  $\omega\to-\infty$. We can observe from \cref{fig:SgMgAsymp_largew_Omg0and0_5} (for $\Omega=0$ and $\Omega=0.5$) that,
\begin{align*}
 |S_g(\Omega,\omega)| \sim |\omega|^3, \quad M_g(\Omega,\omega) \sim |\omega|^2,\quad \mbox{as }\omega\to-\infty.
\end{align*}
The analysis of the two observed asymptotic rates is ongoing. As derived in \cref{subsec:omega infinity},
here we denote the Thomas-Fermi approximation of the action ground state as
	\[ \phi_{\omega}^{\text{TF}}(\mathbf{x}) = |\omega|^{\frac{1}{2}}\tilde{\phi}\left(|\omega|^{-\frac12}\mathbf{x}\right),\quad \tilde{\phi}(\mathbf{x})=\left[\left(1-V(\mathbf{x})\right)_+\right]^{\frac{1}{2}}. \]
	We find that now
	$ \|\phi_{\omega}^{\text{TF}}\|^2_{L^2} =|\omega|^{2}\|\tilde{\phi}\|^2_{L^2},$
	which gives the same asymptotic rate in mass as the above numerical observation for $\omega\to-\infty$. The action value can not be predicted from $ \phi_{\omega}^{\text{TF}}$ due to its non-smoothness in space.
\cref{fig:action_gs_large_omega_alpha_2.pdf} shows a one-dimensional example to illustrate the validity of the Thomas-Fermi approximation, which confirms  (\ref{TF limit}).

\begin{figure}[!t]
		\centering
		\includegraphics[width=.6\textwidth,height=.35\textwidth]{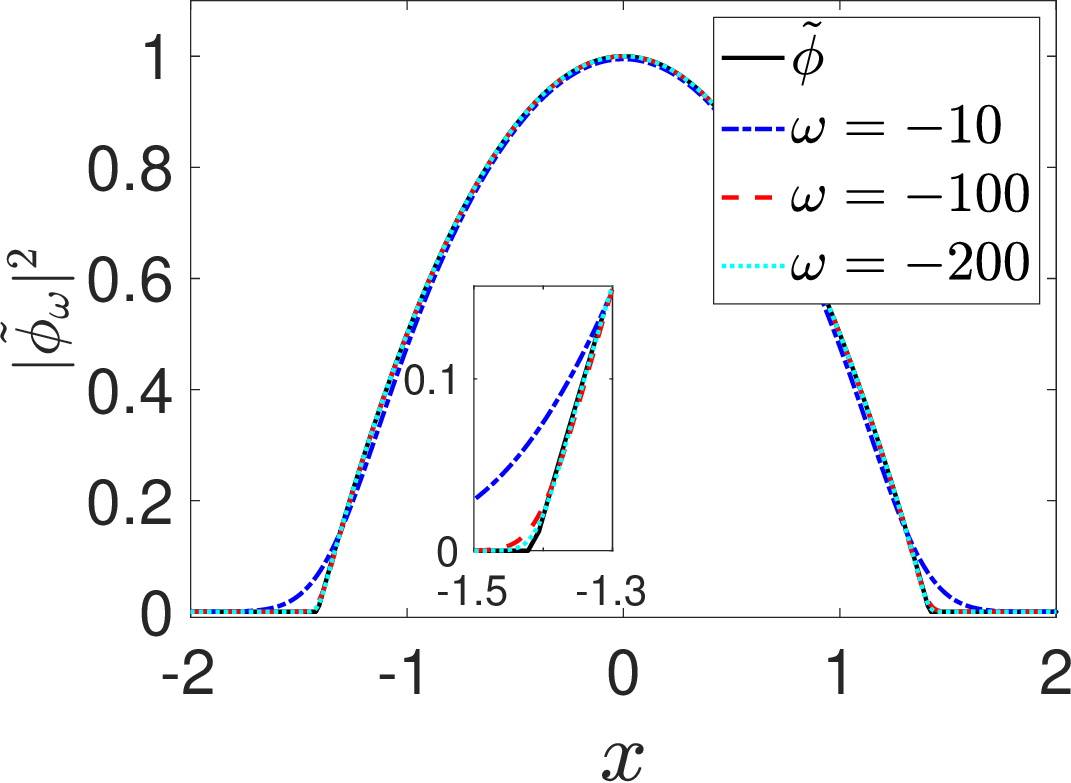}
  \vspace{-2ex}
		\caption{The rescaled action ground state $\tilde\phi_\omega$ (\ref{rescaled state}) for different $\omega$ and the limit $\tilde\phi$.}
  \label{fig:action_gs_large_omega_alpha_2.pdf}
\end{figure}

\begin{figure}[!h]
\centering
\footnotesize
\includegraphics[width=.3\textwidth]{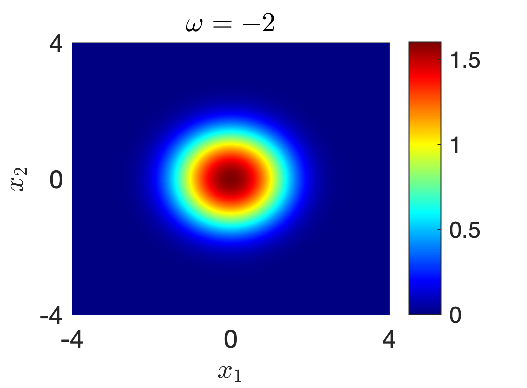}
\includegraphics[width=.3\textwidth]{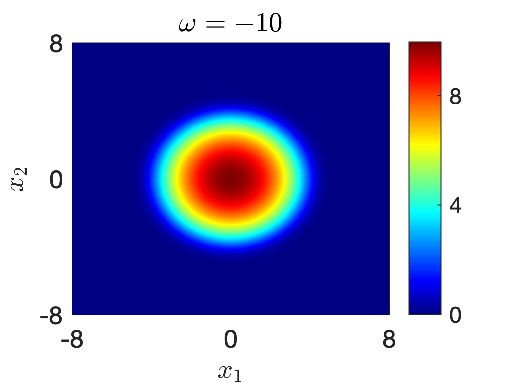}
\includegraphics[width=.3\textwidth]{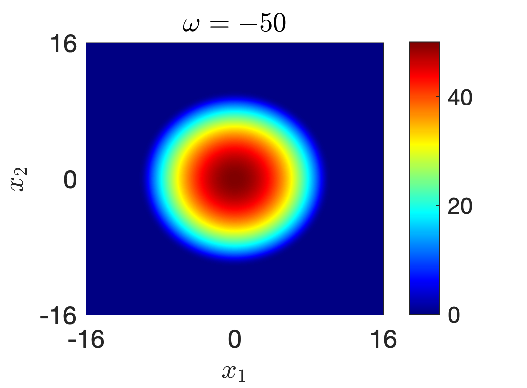} \\[-.5ex]
\includegraphics[width=.3\textwidth]{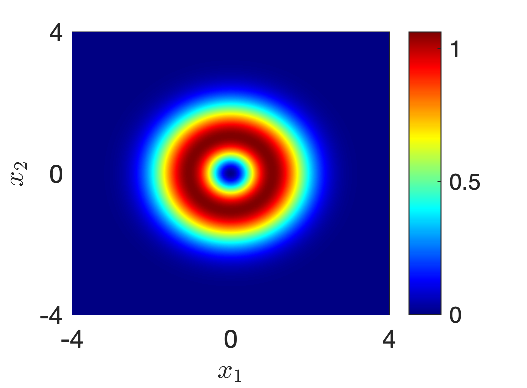}
\includegraphics[width=.3\textwidth]{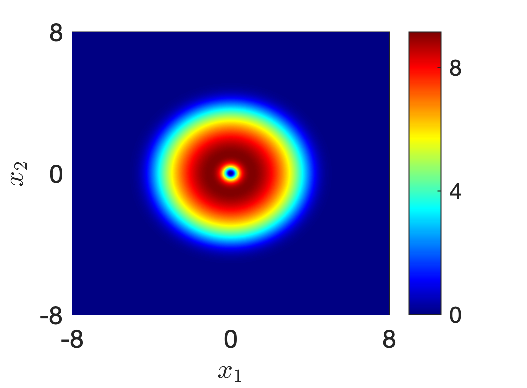}
\includegraphics[width=.3\textwidth]{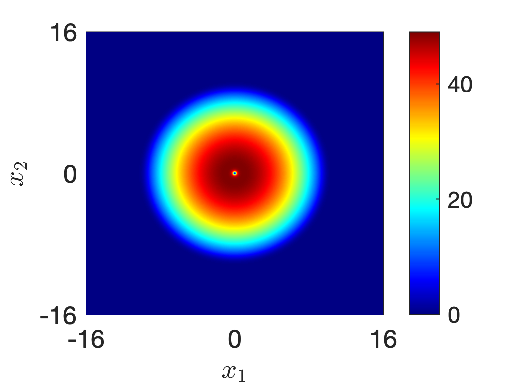} \\
\vspace{-2ex}
\caption{Density profiles of  $|\phi_{\Omega,\omega}(x_1,x_2)|^2$ around the critical $\Omega^c$: $0.766<\Omega^c<0.767$ for $\omega=-2$ (left column), $0.292<\Omega^c<0.293$ for $\omega=-10$ (center column),   $0.089<\Omega^c<0.090$ for $\omega=-50$ (right column), where the lower bounds give non-vortex states (1st row) and the upper bounds give vortex states (2nd row).}
\label{fig:AGSdensity_p3gmx1gmy1_nearOmgc}
\end{figure}
\subsection{Vortices and  critical angular velocity.}\label{sec:Vortices}
The rotational force in (\ref{model}) could induce quantized vortices in the ground state. This distinguishes the defocusing RNLS significantly from the focusing case and is of great interest in various applications. The vortex phenomenon has been well studied in the context of energy ground states\cite{BaoCai}, and now we investigate the counterpart for the action ground states. Besides the usual rotating speed $\Omega$, the chemical potential parameter $\omega$ in fact will also play an important role for the action ground sates in the pattern formation as we shall see below. This will make the vortex phenomena more fruitful. 

Our numerical computations show that for any $\omega<-1$, there exists a critical angular velocity $0<\Omega^c<1$ depending on $\omega$ such that when $0\leq\Omega<\Omega^c$, the density profile of action ground state is unchanged, i.e., $|\phi_{\Omega,\omega}|^2=\phi_{\omega}^2$. It remains  radially symmetric, and $S_g(\Omega,\omega)\equiv S_g(\omega)=S_g(\Omega=0,\omega)$. This is  consistent with \cref{thm:defocusing} and is in fact a stronger result.  When $\Omega>\Omega^c$, a radially symmetric central vortex starts to appear in the action ground state $\phi_{\Omega,\omega}$ and $S_g(\Omega,\omega)<S_g(\omega)$. These findings are demonstrated through
\cref{fig:AGSdensity_p3gmx1gmy1_nearOmgc} and  \cref{fig:SgMgvsOmg_wn5}. The former plots the density of the action ground states around the critical angular velocity $\Omega^c$ for $\omega=-2,-10,-50$, and the latter plots  the change of the ground state action $S_g(\Omega,\omega)$ and  mass $M_g(\Omega,\omega)$ with respect to $\Omega$ under $\omega=-5$.
When $\Omega$ further increases, \cref{fig:AGSdensity_p3gmx1gmy1_wn5} which plots the action ground states for fixed $\omega=-5$ and different $\Omega$, show that a complex vortex lattice pattern is gradually formed with the number of vortices  rapidly  increasing, and the density function becomes no longer radially symmetric.

\begin{figure}[!t]
\centering
\footnotesize
\includegraphics[width=.45\textwidth,height=.35\textwidth]{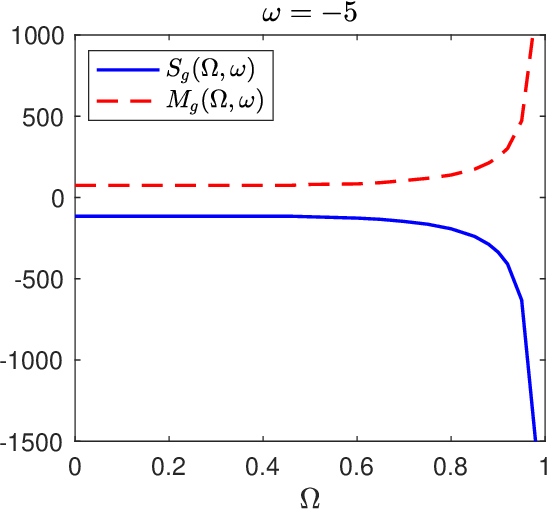}
\vspace{-2ex}
\caption{Change of $S_g(\Omega,\omega)$ and  $M_g(\Omega,\omega)$ with respect to $\Omega$ under   $\omega=-5$.}
\label{fig:SgMgvsOmg_wn5}
\end{figure}
\begin{figure}[!t]
\centering\footnotesize
\includegraphics[width=.3\textwidth]{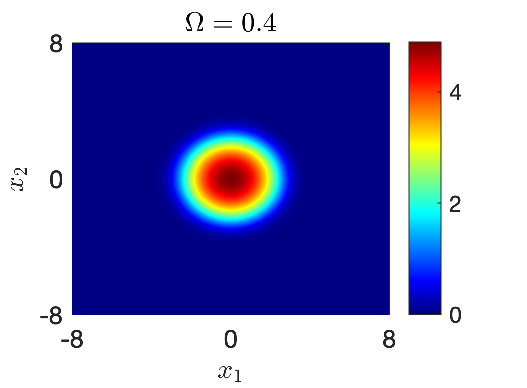}
\includegraphics[width=.3\textwidth]{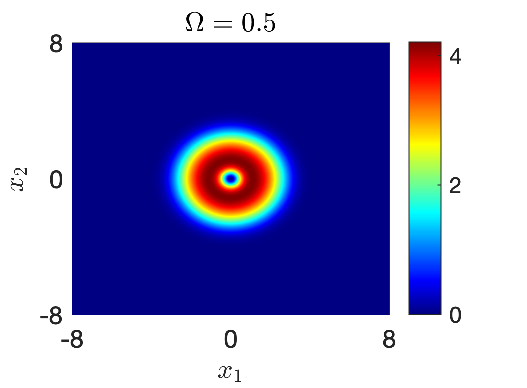}
\includegraphics[width=.3\textwidth]{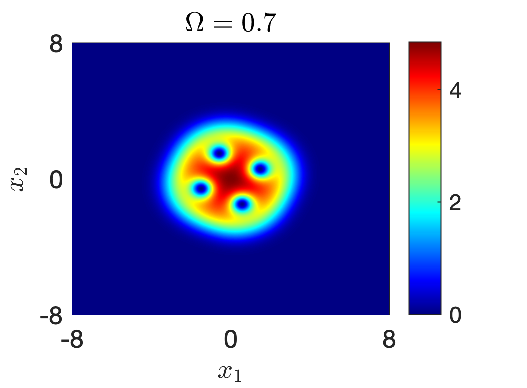} \\
\includegraphics[width=.3\textwidth]{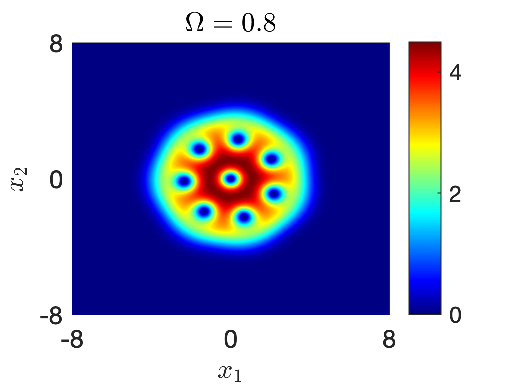}
\includegraphics[width=.3\textwidth]{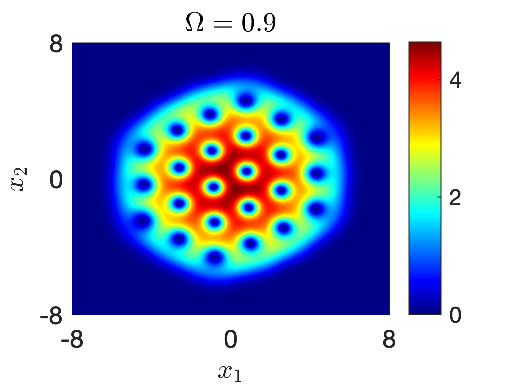}
\includegraphics[width=.3\textwidth]{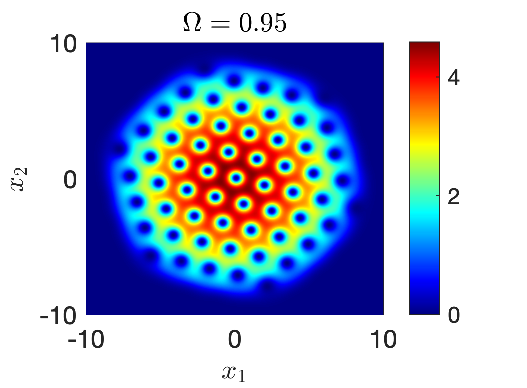} \\
\vspace{-2ex}
\caption{Density profiles of $|\phi_{\Omega,\omega}(x_1,x_2)|^2$ for $\omega=-5$ and different $\Omega$.}
\label{fig:AGSdensity_p3gmx1gmy1_wn5}
\end{figure}
\begin{figure}[!t]
\centering
\footnotesize
\includegraphics[width=.4\textwidth,height=.29\textwidth]{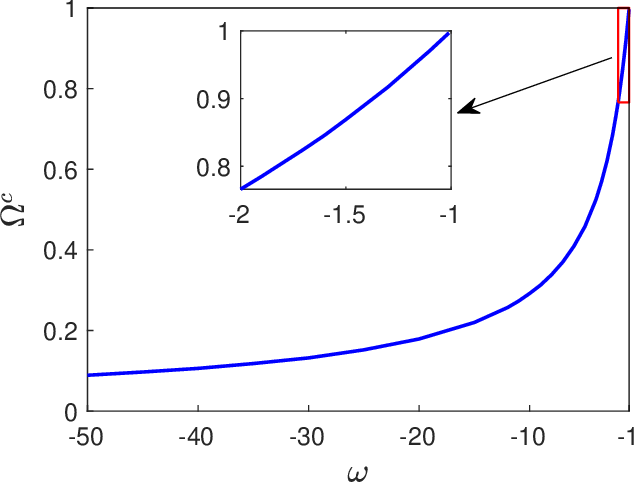}\qquad\quad
\includegraphics[width=.4\textwidth,height=.3\textwidth]{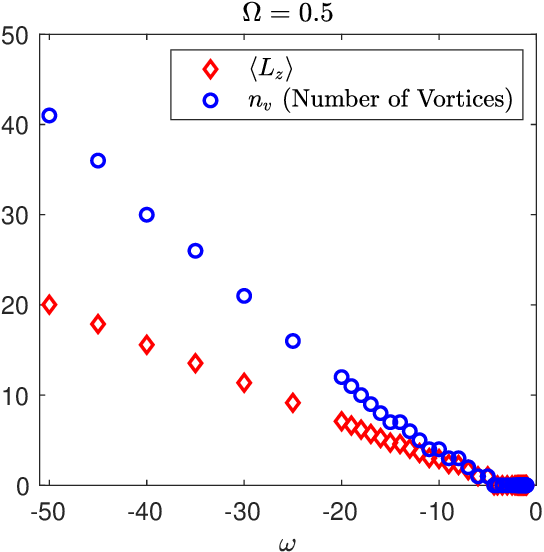}
\vspace{-2ex}
\caption{Critical angular velocity $\Omega^c$ as a function of $\omega$ (left); Number of vortices $n_v$ in $\phi_{\Omega,\omega}$ and the angular momentum expectation $\langle L_z\rangle$ against $\omega$ (right).}
\label{fig:Omgcp3}
\end{figure}

The left plot of \cref{fig:Omgcp3} further addresses  the variation of the critical angular velocity $\Omega^c$ with respect to $\omega$. It can be observed that $\Omega^c$ increases monotonically as $\omega$ increases, and
\[ \lim_{\omega\to-\infty} \Omega^c=0,\qquad \lim_{\omega\to-1^-} \Omega^c = 1. \]
Thus, for $\omega$ in the limit regime $\omega\to-\infty$, a very small $\Omega>0$ can trigger vortices in the action ground states.
To understand the change of vortices in the action ground state with respect to $\omega$, we  fix $\Omega=0.5$ and plot in the right part of \cref{fig:Omgcp3} the number of vortices $n_v$ and the angular momentum expectation, i.e., $\langle L_z\rangle:=\frac{1}{\|\phi_{\Omega,\omega}\|_{L^2}^2}\int_{\mathbb{R}^2}\overline{\phi_{\Omega,\omega}}L_z\phi_{\Omega,\omega}\,\rmd\mathbf{x}$, against several $\omega$. We can see that both $n_v$ and $\langle L_z\rangle$ show an approximately linear increase as $\omega\to -\infty$.

\begin{table}[!t]
\centering
\caption{Quantitative results obtained from the action ground state  for three different $\omega$ and the corresponding energy ground state.}
\label{tab:AGS-then-EGS_Omg0_5_diffw}
\begin{tabular}{ccccccc}
\hline
$\omega$ & $S_g$ & $m$ & $E_g$ & $\mu_g$ & $e^{\text{rel}}_{\omega}$ & $e^{\text{rel}}_S$ \\
\hline
$-30$ &  $-34342.56$ & $3517.33$ &  $71177.33$ & $30.00$ & $3.21\times10^{-10}$ & $7.31\times10^{-14}$ \\
$-40$ &  $-82943.09$ & $6335.55$ & $170478.80$ & $40.00$ & $6.59\times10^{-10}$ & $4.82\times10^{-14}$ \\
$-50$ & $-163950.62$ & $9997.17$ & $335907.93$ & $50.00$ & $2.21\times10^{-10}$ & $2.38\times10^{-14}$ \\
\hline
\end{tabular}
\end{table}
\begin{figure}[!t]
\centering
\footnotesize
\includegraphics[width=.3\textwidth]{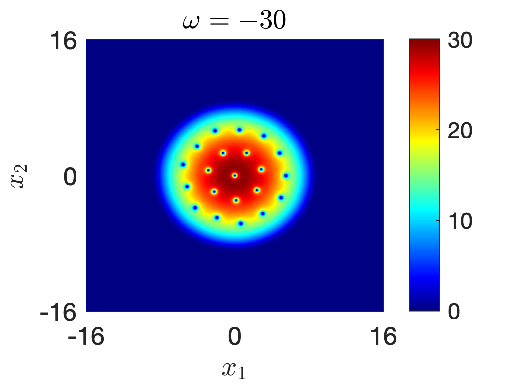}
\includegraphics[width=.3\textwidth]{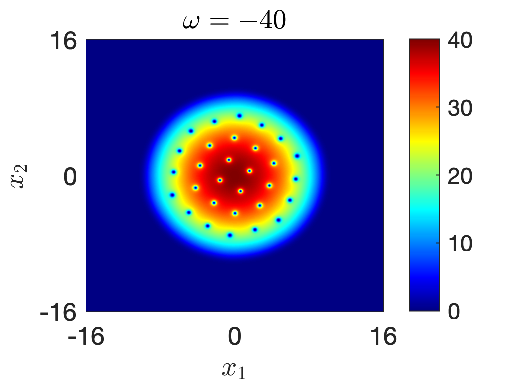}
\includegraphics[width=.3\textwidth]{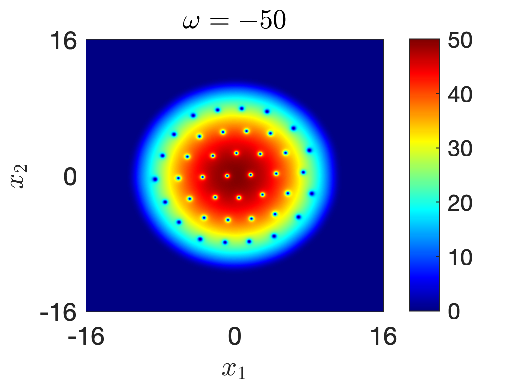} \\
\includegraphics[width=.3\textwidth]{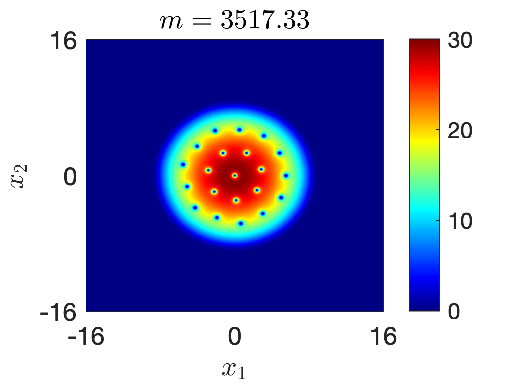}
\includegraphics[width=.3\textwidth]{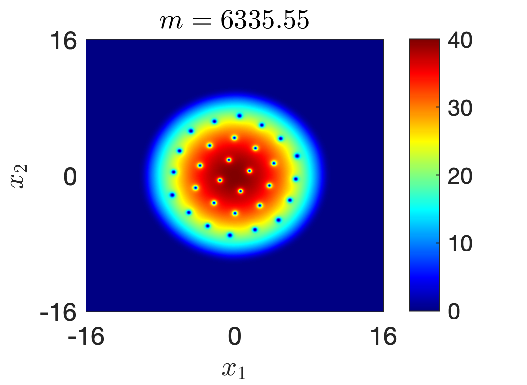}
\includegraphics[width=.3\textwidth]{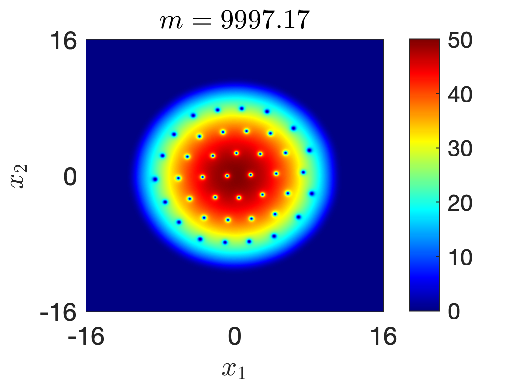} \\
\vspace{-2ex}
\caption{Density profiles of the action ground state $|\phi_{\Omega,\omega}(x_1,x_2)|^2$ (top row) and the corresponding energy ground state $|\phi_m^E(x_1,x_2)|^2$ (bottom row).
}
\label{fig:AGS-EGS-Density_p3gmx1gmy1_wn304050_Omg0_5}
\end{figure}

\subsection{Relation with energy ground state and evidence of non-equivalence}\label{sec:num_non-eq}
Finally, we examine the relation between the action and the energy ground states. Here we fix $\Omega=0.5$.

1) \textit{Conditional equivalence}. For a given $\omega$, we begin by computing the action ground state $\phi_{\Omega,\omega}$ and obtain its action $S_g=S_g(\Omega,\omega)$ and mass $m = M_g(\Omega,\omega)$. We then compute the energy ground state $\phi_m^E$ with the mass $m$ and obtain the corresponding ground state energy $E_g=E(\phi_m^E)$ and chemical potential $\mu_g=\mu(\phi_m^E)$ \eqref{eq:mu_def}. We now compare the obtained results from the action and
the energy ground states. The relative differences $e^{\text{rel}}_{\omega}=|\omega-(-\mu_g)|/|\omega|$ and $e^{\text{rel}}_S=|S_g-(E_g+m\omega)|/|S_g|$ for $\omega=-30,-40,-50$ are listed in \cref{tab:AGS-then-EGS_Omg0_5_diffw}. The density profiles of $\phi_{\Omega,\omega}$ and $\phi_m^E$ are shown in \cref{fig:AGS-EGS-Density_p3gmx1gmy1_wn304050_Omg0_5}.  From \cref{fig:AGS-EGS-Density_p3gmx1gmy1_wn304050_Omg0_5} and \cref{tab:AGS-then-EGS_Omg0_5_diffw}, one can observe that, for these values of $\omega$, the wave function and related physical quantities all form a precise loop from the action ground state to the energy ground state and then to the action ground state, since $\omega\approx-\mu_g$ and $S_g\approx E_g+m(-\mu_g)$ within numerical precision. This matches \cref{thm:actiontoenergy_new}.

2) \textit{Evidence of non-equivalence.} We provide an evidence for the non-equivalence described in \cref{thm:Non-equivalence} by investigating the mass of the action ground state on a finer mesh for $\omega$. As shown in \cref{fig:massAGS_w_Omg0_5} for fixed $\Omega=0.5$, there are a series of critical values of $\omega$, i.e., $\omega_{\mathrm{cr}}^1=-4.3490$, $\omega_{\mathrm{cr}}^2=-6.9025$, $\omega_{\mathrm{cr}}^3=-7.8835$, etc., such that the mass of the action ground state $\|\phi_{\Omega,\omega}\|_{L^2}^2$ has jump discontinuities near these critical values, indicating the existence of multiple action ground states with different mass values at them. More precisely, it is observed that the mass behaves near these critical values as follows:
\begin{align*}
& \|\phi_{\Omega,\omega}\|_{L^2}^2 < 55.4535, \quad \mbox{if }\omega>\omega_{\mathrm{cr}}^1, \quad \|\phi_{\Omega,\omega}\|_{L^2}^2 > 59.6198, \quad \mbox{if }\omega<\omega_{\mathrm{cr}}^1; \\
& \|\phi_{\Omega,\omega}\|_{L^2}^2 < 156.3416,\;\; \mbox{if }\omega>\omega_{\mathrm{cr}}^2, \quad \|\phi_{\Omega,\omega}\|_{L^2}^2 > 161.2327,\;\; \mbox{if }\omega<\omega_{\mathrm{cr}}^2; \\
& \|\phi_{\Omega,\omega}\|_{L^2}^2 < 211.8255,\;\; \mbox{if }\omega>\omega_{\mathrm{cr}}^3, \quad \|\phi_{\Omega,\omega}\|_{L^2}^2 > 216.2082,\;\; \mbox{if }\omega<\omega_{\mathrm{cr}}^3; \quad \cdots,
\end{align*}
and there is no action ground state with a mass value coming from the range
\begin{align}\label{eq:exception-mass}
\big(55.4535,59.6198\big)\bigcup\big(156.3416,161.2327\big)\bigcup\big(211.8255,216.2082\big)\bigcup\cdots.
\end{align}
It is worthwhile to point out that these critical values $\omega_{\mathrm{cr}}^j$ are exactly the transition points where the number of vortices in the action ground state increases. See the density profiles shown in \cref{fig:massAGS_w_Omg0_5}. The first critical value $\omega_{\mathrm{cr}}^1$  corresponds exactly to the critical rotational speed $\Omega^c$ discussed in \cref{sec:Vortices}.

\begin{figure}[!t]
\centering
\footnotesize
\includegraphics[width=.75\textwidth,height=.47\textwidth]{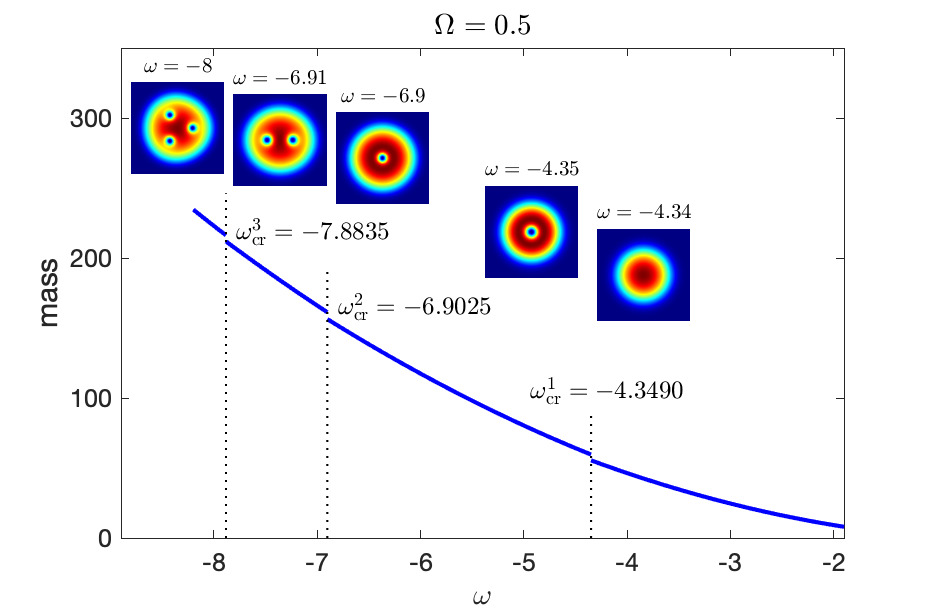}
\vspace{-2ex}
\caption{Change of the mass $\|\phi_{\Omega,\omega}\|_{L^2}^2$ with respect to $\omega$ and density profiles $|\phi_{\Omega,\omega}(x_1,x_2)|^2$ at $\omega=-4.34,-4.35,-6.9,-6.91,-8$ when $\Omega=0.5$.}
\label{fig:massAGS_w_Omg0_5}
\end{figure}
\begin{figure}[!t]
\centering
\footnotesize
\includegraphics[width=.98\textwidth]{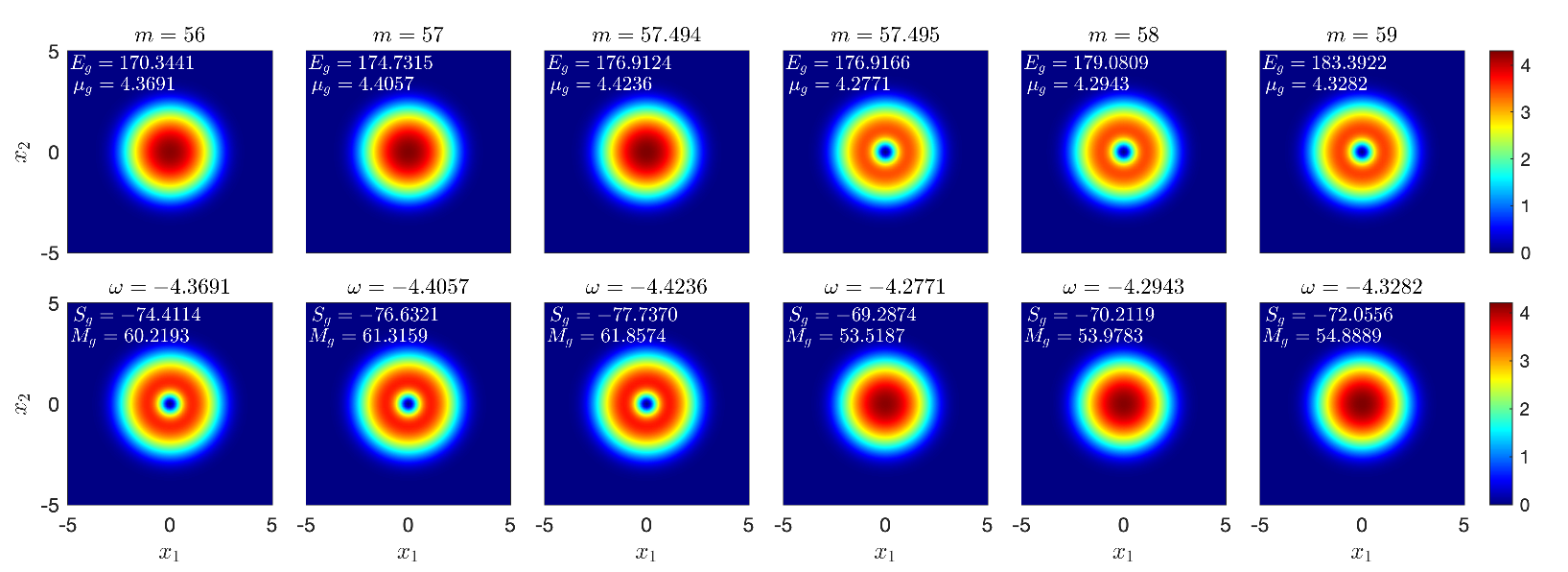}
\vspace{-2ex}
\caption{Density profiles of the energy ground state $|\phi_m^E(x_1,x_2)|^2$ (1st row) and the corresponding action ground state $|\phi_{\Omega,\omega}(x_1,x_2)|^2$ (2nd row) with $\omega=-\mu_g$.}
\label{fig:EGS2AGS_exception_mass}
\end{figure}

As predicted by \cref{thm:Non-equivalence}, any energy ground state $\phi_m^E$ within the given mass range  \eqref{eq:exception-mass} must not be an action ground state. To verify this, \cref{fig:EGS2AGS_exception_mass} presents the numerical results for energy ground states $\phi_m^E$ with several $m\in(55.4535,59.6198)$ and the corresponding action ground states $\phi_{\Omega,\omega}$ with $\omega=-\mu(\phi_m^E)$. The results show that for these values of $m$ (corresponding to $\omega=\omega_{\mathrm{cr}}^1$), the wave function and related physical quantities cannot loop from the energy ground state to the action ground state and then back to the original energy ground state. In addition, we observe that there exists a critical value of mass $m_{\mathrm{cr}}^1\in(57.494,57.495)$ such that, for given $m\in(55.4535,59.6198)$, the chemical potential $\omega=-\mu(\phi_m^E)$ and the mass of corresponding action ground state  satisfy:
\begin{enumerate}[label=(\roman*)]
\item when $m\in(55.4535,m_{\mathrm{cr}}^1)$, $\omega=-\mu(\phi_m^E)<\omega_{\mathrm{cr}}^1$ and $\|\phi_{\Omega,\omega}\|_{L^2}^2>59.6198$;
\item when $m\in(m_{\mathrm{cr}}^1,59.6198)$, $\omega=-\mu(\phi_m^E)>\omega_{\mathrm{cr}}^1$ and $\|\phi_{\Omega,\omega}\|_{L^2}^2<55.4535$.
\end{enumerate}
In  \cref{fig:EGS2AGS_exception_mass}, the energy ground states displayed in the top row, while bearing a resemblance to the action ground states depicted in the bottom row, are in fact not  action ground states. Nevertheless, these energy ground states are still nontrivial solutions to the stationary RNLS \eqref{model0} with $\omega=-\mu(\phi_m^E)$, and their  action values are higher than  the ground states, which indicates that  they are action excited states. For example, for $m=57.494$, from \cref{fig:EGS2AGS_exception_mass}, $S_{\Omega,\omega}(\phi_m^E)=E_g-m\mu_g= 176.9124 - 57.494\times 4.4236 = -77.4181 > S_g = -77.7370$.

3) \textit{Tracing solution branches.}
Finally, we use a continuation approach (with respect to $\omega$) to numerically trace the branches of solutions of the stationary RNLS \eqref{model0}. Four solution branches presented in \cref{fig:massAGS_w_Omg0_5} have been continued with their mass shown in \cref{fig:Mass_SolutionBranches_Omg0_5}. It can be observed that, as $\omega$ changes, these solutions continue to exist when they are no longer the action ground state. More precisely, each solution branch exists for all $\omega$ below a certain threshold which is getting smaller and smaller for solutions with an increasing number of vortices. For example, the branch of the solution without vortex exists for all $\omega<-1$ (see the blue curve in \cref{fig:Mass_SolutionBranches_Omg0_5}), while the branch of the solution with a central vortex exists only for all $\omega<-1.5$ (see the red curve in \cref{fig:Mass_SolutionBranches_Omg0_5}). These numerical observations in fact further demonstrate the multi-solution structure of the stationary RNLS \eqref{model0} for relative smaller $\omega$. In particular, the threshold of $\omega$ that makes the occurrence of the central vortex state solution estimates the critical value for the existence of multiple nontrivial solutions. As $\omega$ continues to decrease, more and more multi-vortex solutions emerge and the action ground state continues to be given by the solutions with more vortices.

\begin{figure}[!t]
\centering
\footnotesize
\includegraphics[width=.75\textwidth,height=.45\textwidth]{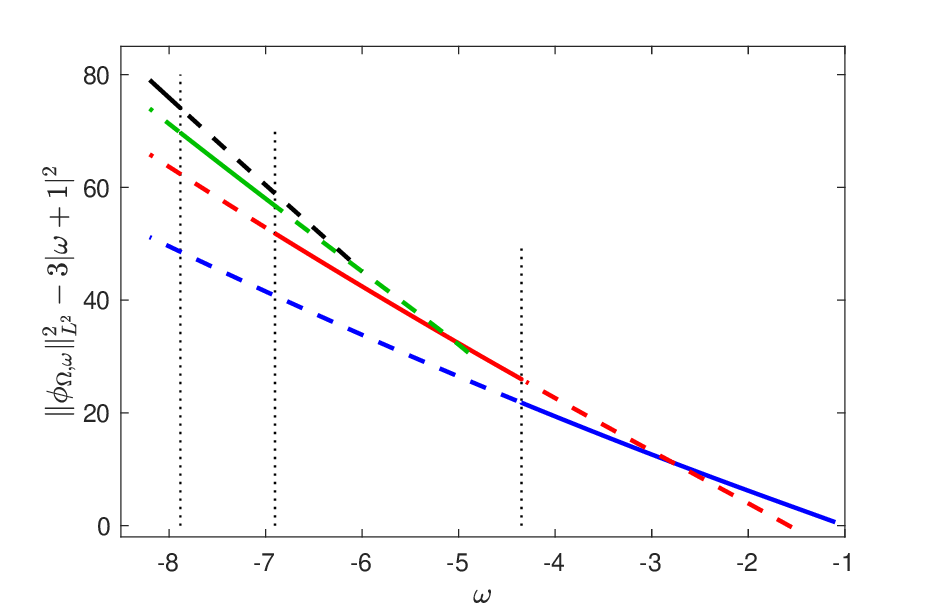}
\vspace{-2ex}
\caption{Change of the mass (shown in a relative value $\|\phi_{\Omega,\omega}\|_{L^2}^2-3|\omega+1|^2$) of the solution branches presented in \cref{fig:massAGS_w_Omg0_5} with respect to $\omega$ when $\Omega=0.5$. Curves colored blue, red, green, and black correspond to solution branches with 0, 1, 2, and 3 vortices, respectively. The solid part of each curve indicates that the corresponding solutions are the action ground state, and the dashed part indicates that they are no longer the action ground state.}
\label{fig:Mass_SolutionBranches_Omg0_5}
\end{figure}

\section{Conclusion}\label{sec:con}
In this study, we examined the action ground states of the rotating nonlinear Schr\"odinger equation in the defocusing nonlinear interaction case. The investigations encompass both theoretical and numerical aspects. The theoretical studies yielded qualitative and quantitative features of the action ground states. A specific analysis was conducted to understand the relationship between the action ground states and the more commonly studied energy ground states. In the absence of a rotating force, i.e., the non-rotating case, the complete equivalence between the two types of ground states was established. For the rotating case, the conditional equivalence was proven, and a characterization for the possible occurrence of non-equivalence was derived. Alongside the analysis, the asymptotic behaviors of the mass and action at the ground state solution were obtained in certain physically relevant parameter regimes.
The numerical investigation conducted experiments to validate and complement the theoretical studies. Particular emphasis was given to investigating vortex phenomena in action ground states. Notably, the number of quantized vortices in the ground state solution was found to be significantly influenced by the given chemical potential parameter. Building on this finding and the theoretical non-equivalence characterization, we designed numerical experiments to observe the non-equivalence between the action ground state and the energy ground state.

\appendix

\section{A fact of ground-state energy under symmetric harmonic oscillator}

\begin{lemma}\label{conj1-lambdaOmg}
If $V(\mathbf{x})=\frac12\gamma_r^2|\mathbf{x}|^2$ ($d=2$), and $|\Omega|<\gamma_r$, then $\lambda_0(\Omega)\equiv\lambda_0(0)=\gamma_r$ with the corresponding eigenfunction $\varphi_0=\sqrt{\gamma_r/\pi}\mathrm{e}^{-\gamma_r|\mathbf{x}|^2/2}$ up to a constant phase factor. If $V(\mathbf{x})=\frac12\gamma_r^2(x_1^2+x_2^2)+\frac12\gamma_3x_3^2$ ($d=3$), $\gamma_3>0$, and $|\Omega|<\gamma_r$, then $\lambda_0(\Omega)\equiv\lambda_0(0)=\gamma_r+\frac12\gamma_3$ with the corresponding eigenfunction $\varphi_0=\frac{\gamma_r^{1/2}\gamma_3^{1/4}}{\pi^{3/4}}\mathrm{e}^{-\gamma_r(x_1^2+x_2^2)/2-\gamma_3x_3^2/2}$ up to a constant phase factor.
\end{lemma}

\begin{proof}
Direct computations show that, in 2D polar coordinate $(r,\theta)$, all the eigenvalues and corresponding eigenfunctions of $R=-\frac{1}{2}\Delta+\frac{1}{2}\gamma_r^2|\mathbf{x}|^2-\Omega L_z$ are given as\cite{BaoLiShen09}: for $n\in\mathbb{N},\, m\in\mathbb{Z},$
\[ \lambda_{mn}(\Omega)=(2n+|m|+1)\gamma_r-m\Omega,\quad \varphi_{mn}(r,\theta)=C_{mn}r^{|m|}\mathcal{L}_n^{(|m|)}(\gamma_r r^2)\mathrm{e}^{-\gamma_r r^2/2}\mathrm{e}^{im\theta},\]
where $\mathcal{L}_n^{(|m|)}$ are generalized Laguerre polynomials and $C_{mn}>0$ are normalization constants such that $\|\varphi_{mn}\|_{L^2}=1$. Since $|\Omega|<\gamma_r$, the smallest eigenvalue is $\lambda_0(\Omega)=\lambda_{00}(\Omega)=\gamma_r$, which is single-fold, with the corresponding eigenfunction $\varphi_{00}(r,\theta)=\sqrt{\gamma_r/\pi}\mathrm{e}^{-\gamma_r r^2/2}$. The proof in 3D can be done by using the cylindrical coordinate $(r,\theta,x_3)$ and noting that $-\frac12\partial_{x_3}^2+\frac12\gamma_3^2x_3^2$ has eigenvalues $\left(k+\frac12\right)\gamma_3$ ($k\in\mathbb{N}$) with corresponding eigenfunctions given by scaled Hermite functions\cite{BaoLiShen09}.
\end{proof}

Note that $\lambda_0(\Omega)$ can be viewed as the ground-state energy of a non-interacting rotating BEC\cite{BaoCai}. Lemma~\ref{conj1-lambdaOmg} says that under a 2D radially symmetric or 3D cylindrical symmetric harmonic oscillator potential, the ground state is independent of the rotational speed $\Omega$.

\section*{Acknowledgment}
C. Wang is supported by the Ministry of Education of Singapore under its AcRF Tier 2 funding MOE-T2EP20122-0002 (A-8000962-00-00). W. Liu is supported by National Natural Science Foundation of China 12101252 and the Guangdong Basic and Applied Basic Research Foundation 2022A1515010351 through South China Normal University. X. Zhao is supported by the National Natural Science Foundation of China 12271413. Part of this work was done when the authors were visiting Institute for Mathematical Sciences, National University of Singapore in spring of 2023.


\end{document}